\documentclass[a4paper,11pt,reqno]{article}
\usepackage[english]{babel}
\usepackage{cite}
\usepackage{amsfonts}
\usepackage{latexsym}
\usepackage{amsthm}
\usepackage{amsopn}
\usepackage{amsmath}
\usepackage[all]{xy}
\usepackage{verbatim}
\usepackage{amssymb}
\usepackage{mathrsfs}
\usepackage{float}
\usepackage[mathscr]{eucal}
\usepackage[active]{srcltx}
\usepackage{fullpage}
\usepackage{hyperref}
\usepackage{manyfoot}
\usepackage{amscd}
\usepackage[dvips]{graphics}
\usepackage[dvips]{graphicx}
\usepackage{amscd}
\usepackage{color}
\usepackage{cite}
\usepackage{tikz}
\usepackage{pgfplots}
\pgfplotsset{compat=1.15}
\usepackage{mathrsfs}
\usetikzlibrary{arrows}

\title{Note on a family of surfaces with $p_g=q=2$ and $K^2=7$}
\author{Matteo Penegini and Roberto Pignatelli}
\date{}

\newcommand\Pic{\text{\rm Pic}}

\newcommand\OO{{\mathcal{O}}}

\newtheorem{thm}{Theorem}[section]
\newtheorem{lem}[thm]{Lemma}
\newtheorem{cor}[thm]{Corollary}
\newtheorem{prop}[thm]{Proposition}

\theoremstyle{definition}
\newtheorem{defin}[thm]{Definition}

\newtheorem{rem}[thm]{Remark}
\theoremstyle{remark}

\newcommand{\ZZ}{\mathbb{Z}}

\newcommand{\CC}{\mathbb{C}}
\newcommand{\PP}{\mathbb{P}}

\newcommand{\lr}{\longrightarrow}
\newcommand{\oo}{\mathcal{O}}
\newcommand{\mL}{\mathcal{L}}
\newcommand{\mQ}{\mathcal{Q}}
\newcommand{\mT}{\mathcal{T}}
\begin{document}
\maketitle
\begin{center}

To Professor F. Catanese on the occasion of his 70th birthday
\end{center}
\begin{abstract} 
We study a family of surfaces of general type with  $p_g=q=2$ and $K^2=7$, originally constructed by C. Rito in \cite{rito}. We provide an alternative construction of these surfaces, that
allows us to describe their Albanese map and the corresponding locus $\mathcal{M}$ in the moduli space of the surfaces of general type. In particular we prove that $\mathcal{M}$ is an open subset, and it has three connected components, two dimensional, irreducible and generically smooth.
\end{abstract}

\Footnotetext{{}}{\textit{2020 Mathematics Subject Classification}:
14J29, 14J10, 14B12}

\Footnotetext{{}} {\textit{Keywords}: Surface of general type,
Albanese map, Moduli spaces}

\Footnotetext{{}} {\textit{Version:} 13 Oct. 2020}
\section{Introduction}

In the last two decades, several authors worked intensively on the classification of irregular algebraic
surfaces (i.e., surfaces S with $q(S) > 0$) and produced a considerable amount of results,
see for example the survey papers \cite{BaCaPi06,MP12,P} for a detailed bibliography on the
subject.

In particular, irregular surfaces of general type with $\chi(\mathcal{O}_S) = 1$, that is, $p_g(S) = q(S) \geq 1$ were
investigated.  By, nowadays classical, Debarre inequality [Deb81, Th\'eor\`eme 6.1] we have $p_g \leq 4$. Surfaces with
$p_g = q = 4$ and $p_g = q = 3$ are completely classified, see \cite{Be82, CaCiML98,
HP02,Pir02}. On the other hand, for the the case $p_g = q = 2$, which presents a very rich
and subtle geometry, we have so far only a partial understanding of the situation; we refer
the reader to \cite{Ca00,Ca11,Ca15,Pe09,PP, PP10, PePol14, PiPol17,PRR20, Z03} for an account on this topic and recent
results.

As the title suggest, in this paper we consider a family of minimal surfaces of general
type with $p_g = q = 2$ and $K^2 = 7$. The existence of these surfaces was originally established
by Rito in \cite{rito}; the
present work provides an alternative construction of them, that allows us
to describe their Albanese map and their moduli space.

Our results can be summarized as follows.

\begin{thm}\label{Theorem A} The Gieseker moduli space $\mathcal{M}^{\mathrm{can}}_{2, \, 2, \,7}$ of the canonical models of the surfaces of general type with $p_g=q=2$ and $K^2=7$ contains three pairwise disjoint open subsets, 
all irreducible, generically smooth of dimension $2$, such that
for each surface $S$  
in them, the Albanese map 
is a generically finite double cover onto a $(1,2)$-polarized non simple abelian surface $A$.
\end{thm}

 It is worth to notice here, that there is only another known family of surfaces of general type with  $p_g = q = 2$ and $K^2 = 7$ found by Cancian and Frapporti in \cite{CanFr15} and described in details in \cite{PiPol17} whose elements have a different Albanese map. Namely, the Albanese map is a generically finite \emph{triple} cover of a principally polarized abelian surface. Hence, being the degree of the Albanese map a topological invariant (see Proposition \ref{prop.degree.alb}),  these families  provide a substantially new piece in the fine classification of minimal surfaces of general type with $p_g=q=2$ in the spirit of \cite{Ca84,Ca89,Ca90}.

\medskip

The paper is organized as follows.

In Section \ref{Sec_theConstruction} we explain our construction in details, pointing out the similarities and the differences with \cite{rito}, and computing the invariants of the resulting surfaces (Proposition \ref{prop_invariantiRito}). We study their Albanese map, giving a precise description of its image, isogenous to a product of two curves of genus $1$,  and of its branch curve.

In Section \ref{sec_ellfib} we use our description to study the modular image of Rito's family, showing that it has three  connected components, all irreducible of dimension $2$. 

The last two sections contain results of deformation theory headed to compute $h^1(S, \, T_S)=2$ (Proposition \ref{prop_H1tan}) from which it follows that each component is open and generically smooth in the moduli space. 

Section \ref{sec_deform} is devoted to a general result, Theorem \ref{thm_def}, about the deformations of the blow up in a point, that was crucial for the proof and that we find of independent interest. The situation is the following: consider a point $p$ in a smooth surface $B$, a curve $D$ in $B$ smooth at $p$ and a vector $v \in T_pB$. A standard exact sequence associates to $v$ an infinitesimal deformation $\mathcal{B}$ of the blow-up of $B$ in $p$. Then Theorem \ref{thm_def} says that $\mathcal{B}$ contains an infinitesimal deformation of the strict transform of $D$ if and only if the class of $v$ in the normal vector space $T_pB/T_pD$ extends to a global section of the normal bundle of $D$ in $B$.

Finally Section \ref{moduli} is devoted to the study of the first-order deformations of the surfaces in $\mathcal{M}$. To show $h^1(S, \, T_S)=2$,
 we show in fact that the map $H^1(S, \, T_S) \rightarrow H^1(A, \, T_A)$ is injective, and its image is given by the infinitesimal deformations of $A$ that are still isogenous to a product.

\medskip

\textbf{Acknowledgments.} Both authors were partially supported by
GNSAGA-INdAM.

\medskip

\textbf{Notation and conventions.} We work over the field
$\mathbb{C}$ of complex numbers. 
By \emph{surface} we mean a projective, non-singular surface $S$,
and for such a surface $K_S$ denotes the canonical
class, $p_g(S)=h^0(S, \, K_S)$ is the \emph{geometric genus},
$q(S)=h^1(S, \, K_S)$ is the \emph{irregularity} and
$\chi(\mathcal{O}_S)=1-q(S)+p_g(S)$ is the \emph{Euler-Poincar\'e
characteristic}.
\section{The construction}\label{Sec_theConstruction}
	
In this section we give an alternative, but equivalent,  construction to the surface $S$ of general type with $p_g=q=2$ and $K^2=7$ constructed by Rito in \cite{rito}.

\begin{figure}
\centering
\includegraphics[width=0.75\linewidth]{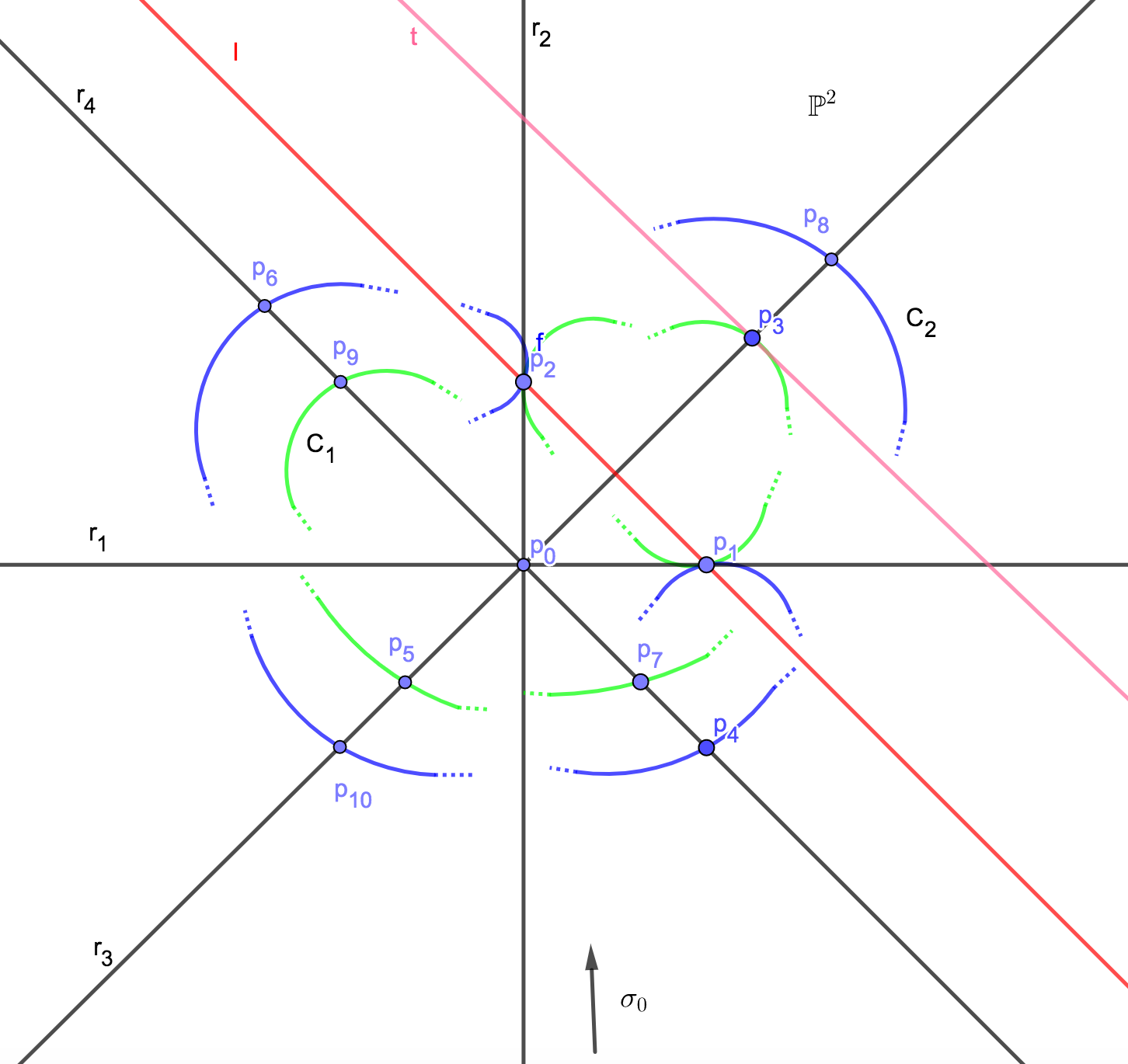}
\includegraphics[width=0.75\linewidth]{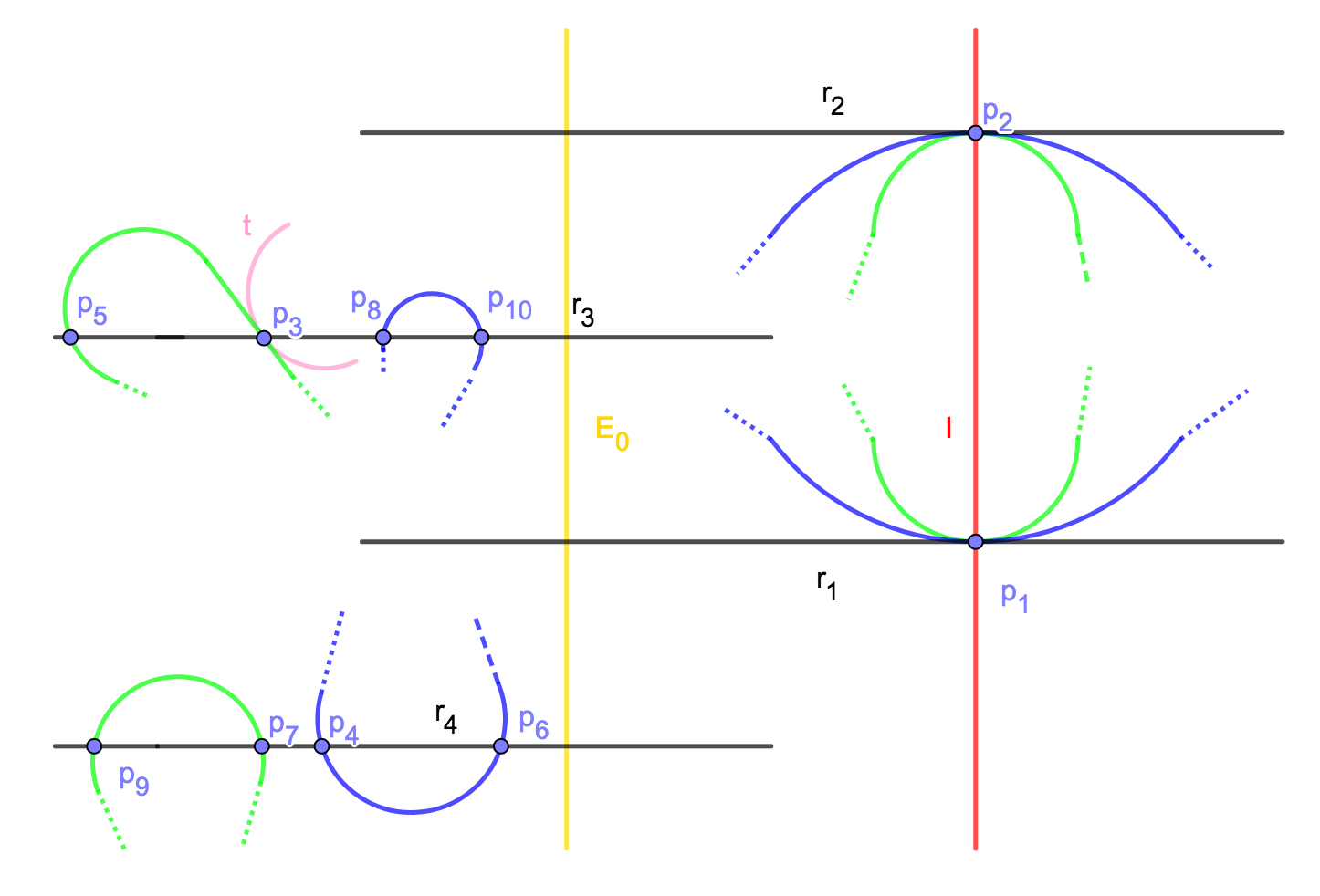}
\caption{$\sigma_0 \colon Bl_{p_0}(\mathbb{P}^2) \longrightarrow \mathbb{P}^2$}\label{Fig1}
\end{figure}

\medskip

Let us fix the following points on $\mathbb{P}^2$ :
\[ p_0=(1:0:0), \, p_1=(0:1:0), \, p_2=(0:0:1), \, p_3=(1:1:1), \, p_4=(1:a:b).
\]
Moreover, let us denote by $r_i$ with $i=1,\ldots ,4$
 the four lines joining $p_0$ with each $p_i$ resectively, i.e.,
 \[
 r_1=(x_2), \, r_2=(x_1), \, r_3=(x_1-x_2), \, r_4=(bx_1-ax_2),
 \]
  and the two conics:
  \[ C_1=(x_0^2-x_1x_2), \quad C_2=(abx_0^2-x_1x_2).
  \]
  Note that both conics are tangent to $r_1$ and $r_2$ respectively in $p_1$ and $p_2$. Finally, $p_3 \in C_1$ and $p_4 \in C_2$.
  
Fix a square root $c$ of $ab$ and consider the following points on the curves we have just defined 
\[ p_5=(1:-1:-1), \, p_6=(1:-a:-b), \, p_7,p_9=(\pm c:a:b), \, p_8,p_{10} = (\pm 1:c:c).
\]  
Finally, let $\ell =(x_0)$ be the line through $p_1$ and $p_2$ and $t=(2x_0-x_1-x_2)$ be the tangent line to $C_1$ through $p_3$, see Figure \ref{Fig1} to have a visual representation of the situation.

Up to now, we followed Rito in \cite{rito}, changing the notation only for the curve $t$ ($R$ in Rito's notation). Now, we proceed a bit differently. 
Let us apply the following birational transformations of $\mathbb{P}^2$:
\begin{enumerate}
\item We blow up the point $p_0$ and we get $\sigma_0 \colon Bl_{p_0}(\mathbb{P}^2) \longrightarrow \mathbb{P}^2$ with exceptional divisor $E_0$ (see Figure \ref{Fig1} again). 

\begin{figure}
\centering
\includegraphics[width=0.75\linewidth]{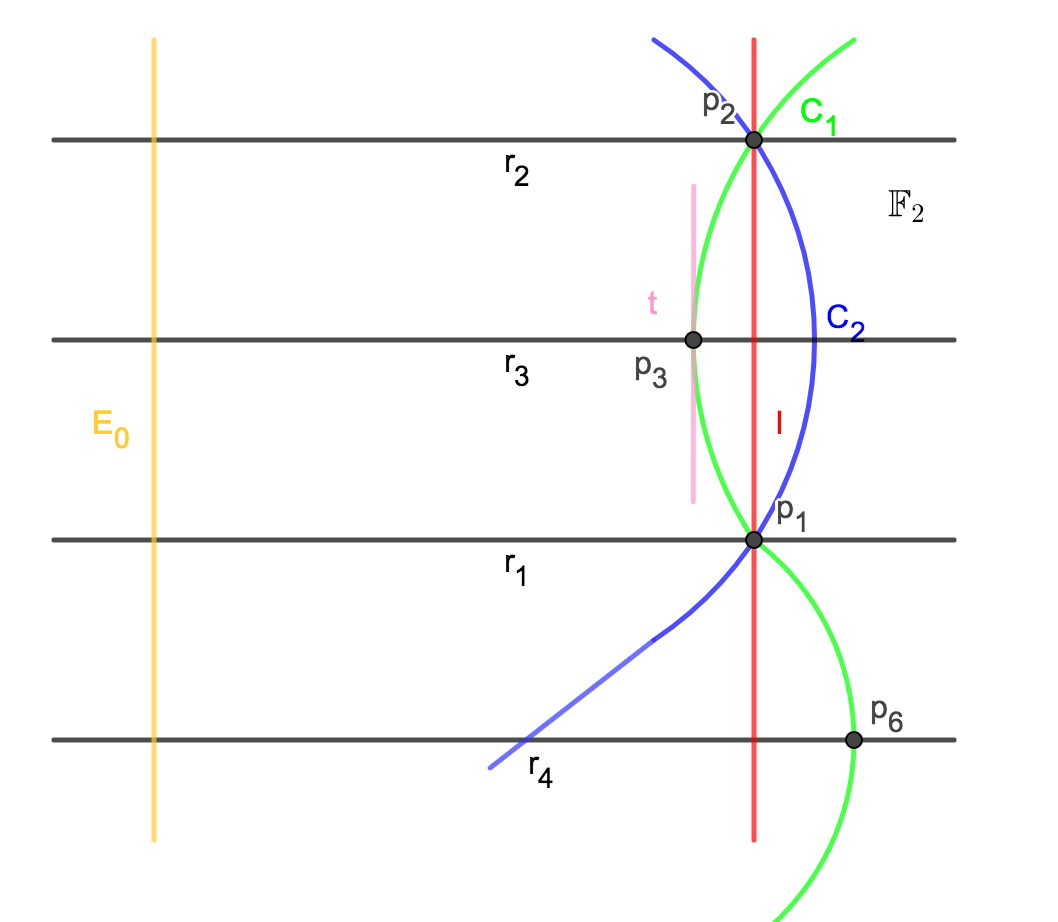}
\includegraphics[width=0.75\linewidth]{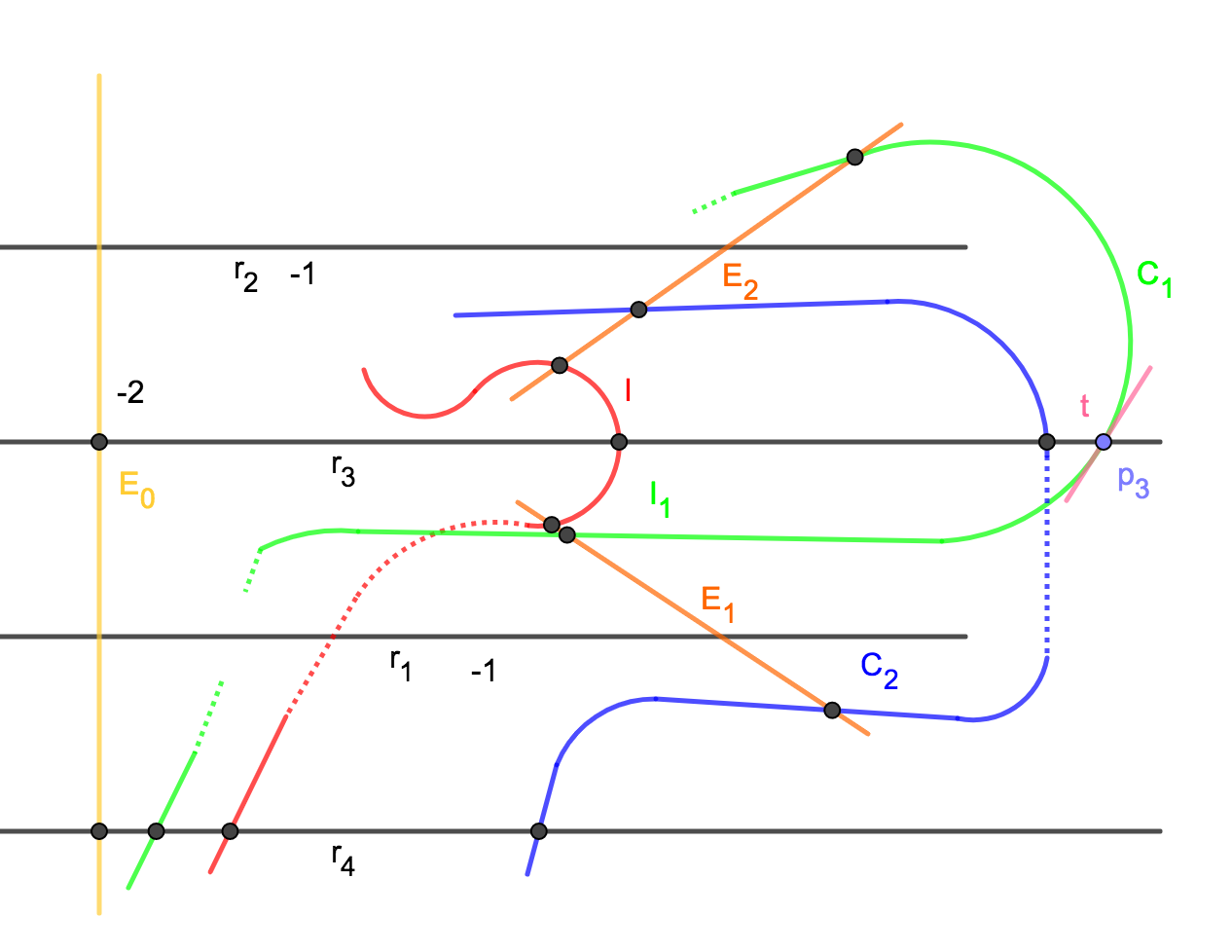}
\caption{The blow up on $\mathbb{F}_2$ of the images of the points $p_1$ and $p_2$}\label{Fig2}
\end{figure}
Considering the pencil of lines through $p_0$ on $Bl_{p_0}(\mathbb{P}^2)$ we have a rational pencil of curves with self-intersection $0$, which include the strict transforms of the four lines $r_i$, $i=1, \ldots ,4$. We notice that on  $Bl_{p_0}(\mathbb{P}^2)$ we can lift the natural involution on $\PP^2$
\[ j \colon (x_0 : \, x_1: \, x_2) \mapsto (-x_0 : \, x_1: \, x_2)
\]
which has as fixed divisor $E_0 + \sigma_0^*(\ell)$.

\item The quotient by this involution $Bl_{p_0}(\mathbb{P}^2)/ j$ is the Segre--Hirzebruch surface $\mathbb{F}_2$. 

The images of the four lines $r_i$ are fibres of the fibration on $\mathbb{F}_2$. Moreover, the only negative section of this fibration coincide with the image of $E_0$.

\item We blow up on $\mathbb{F}_2$ the images of the points $p_1$ and $p_2$, introducing two exceptional divisors $E_1$ and $E_2$.

We recall that the images of the lines $r_1$ and $r_2$ and of the conics $C_1$ and $C_2$ pass all through these points. Performing this operation the images of $r_1$ and $r_2$ became $-1$-curves (see Figure \ref{Fig2}).

\begin{figure}[h]
\centering
\includegraphics[width=0.75\linewidth]{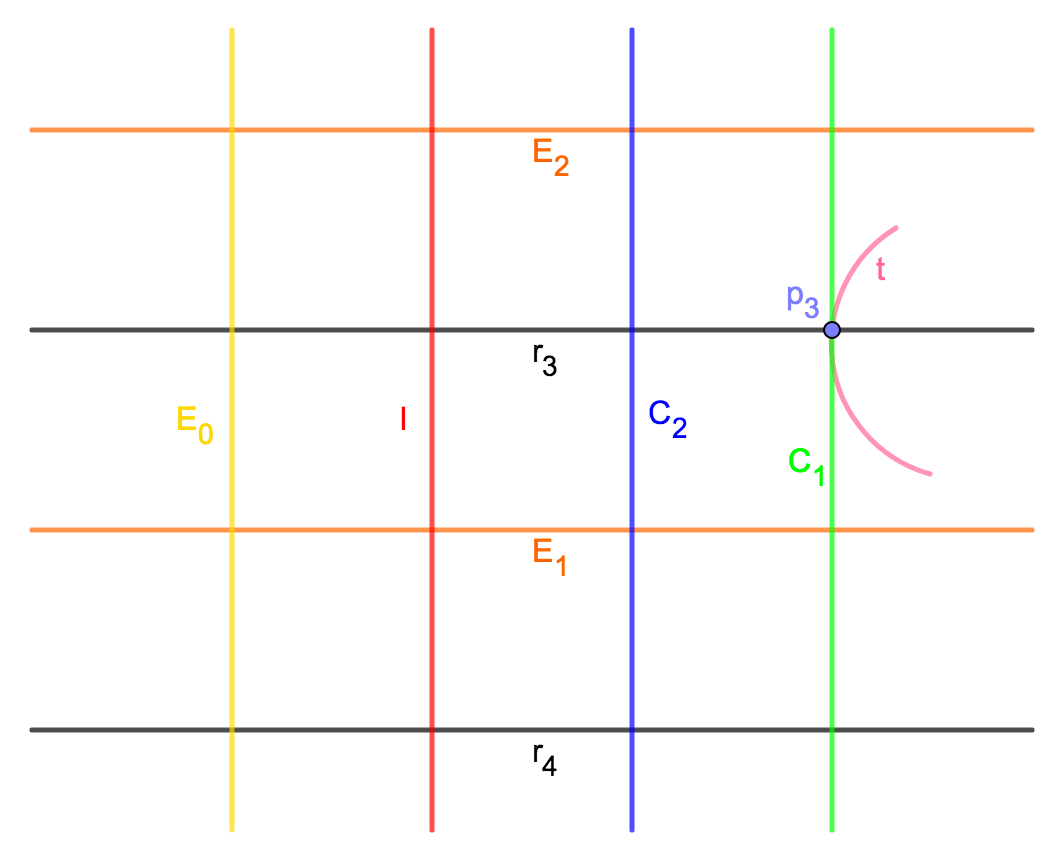}
\caption{Contracting $r_1$ and $r_2$ we find $\PP^1 \times \PP^1$}\label{Fig3}
\end{figure}
\item We contract the images of the curves $r_1$ and $r_2$. The resulting surface is exactly $\PP^1 \times \PP^1$.
\end{enumerate}
Summarizing we have obtained a rational map of degree $2$
\[ \sigma \colon \mathbb{P}^2 \dashrightarrow \mathbb{P}^1 \times \mathbb{P}^1.
\]

We denote with the same letters the strict transform on $\mathbb{P}^1 \times \mathbb{P}^1$ of all the curves considered on $\PP^2$, since no confusion arises (see Figure \ref{Fig3}). 
The bidouble cover of $ \mathbb{P}^1 \times \mathbb{P}^1$ with ramification divisors
\[
D_1 = 0, \,  D_2=E_1+E_2+r_3+r_4, \, D_3=C_1+C_2+E_0+l,
\]
is obviously the product $T_1 \times T_2$ of two double covers $\phi_j \colon T_j \rightarrow \PP^1$ branched at $4$ points, two curves of genus $1$
(see Figure \ref{Fig3}).

The  fibre product of the bidouble cover $T_1 \times T_2 \rightarrow \PP^1 \times \PP^1$ with $\sigma_0$ gives the bidouble cover of $\PP^2$ studied in \cite{rito}[Section 3, Step 1], where it is shown that it is birational to an abelian surface that we denote by $A$ (it was $V'$ in  \cite{rito}).
We can summarize this construction with the following diagram.

  \[
\begin{xy}
\xymatrix{
\mathbb{P}^2   \ar@{.>}[d]_{2:1}   &  A \ar@{->}[d]_{\iota}^{2:1} \ar@{.>}[l]_{\pi}^{2^2:1} \\
\mathbb{F}_2  \ar@{.>}[d]  &  T_1 \times T_2   \ar@{->}[dl]^{2^2:1} \\ 
\PP^1 \times \PP^1
 }
\end{xy}
\]
Note that the map $\iota \colon A \rightarrow T_1 \times T_2$ is an isogeny of degree $2$.

We see (compare \cite{rito}[Section 3, Step 2]) that the strict transform of the curve $C_1$ is tangent to  the curve $t$ on $A$ at a point $p$. This point is a tacnode (singularity of type $(2,2)$) for the strict transform of the curve $t$.
So the divisor $t+ C_1$ is reduced and has a singularity of type $(3,3)$.

\begin{rem} We see that we recover the construction due to Rito \cite{rito} of an abelian surface with a $(1,2)$-polarization, Please notice that in \cite{rito} the abelian surface $A$ was labelled by $V'$ and the curves $C_1$ and $t$ by $\hat{C}_1$ and $\hat{R}$. 
\end{rem}

In \cite{rito} it is shown that the divisor  $t+ C_1$ is even, i.e. there is a divisor $L$ such that
$$t+ C_1\equiv 2L,$$
and that $$( t+ C_1)^2=16.$$\\
So $L$ is a polarization of type $(1,2)$. This is exactly the situation described by the first author and F. Polizzi in \cite[Remark 2.2]{PP}.  There the authors  suggest how to construct a surface with $p_g=q=2$ and $K^2_S=7$ as a generically finite double cover of $A$ branched along a divisor as $t+C_1$.  We follow the suggestion slavishly and we summarize the situation with the following special case of  \cite[Proposition 1]{rito}

\begin{prop}\label{prop_invariantiRito}
Let $A$ be an Abelian surface.
Assume that $A$ contains a reduced curve $t+C_1$ and a divisor $L$ such that $t+C_1 \equiv 2L$,
$(t+C_1)^2=16$ and $t+C_1$ contains a $(3,3)$-point and no other singularity.
Let $S$ be the smooth minimal model of the double cover of $A$ with branch locus $t+C_1$.
Then $p_g(S)=q(S)=2$ and $K_S^2=7$.
\end{prop}

Let us now construct $S$ step by step starting from $A$.
\begin{enumerate}
\item First, we resolve the singularity in $p$. To do that, we need to blow up $A$ twice, first  in $p$ and then in a point infinitely close to $p$. Let us denote these two blow ups by 
\[ B' \stackrel{\sigma_4}{\longrightarrow} B \stackrel{\sigma_3}{\longrightarrow} A.
\]
On $B'$, let us denote by $F$ the exceptional divisor relative to $\sigma_4$, by $E'$ the strict transform of the exceptional divisor $E$ relative to $\sigma_3$, by $C_1$ the strict transform of $C_1$ and, finally, by $R$ the strict transform of $t$ (see Figure \ref{Fig4}). 

\begin{figure}[h]
\centering
\includegraphics[width=0.75\linewidth]{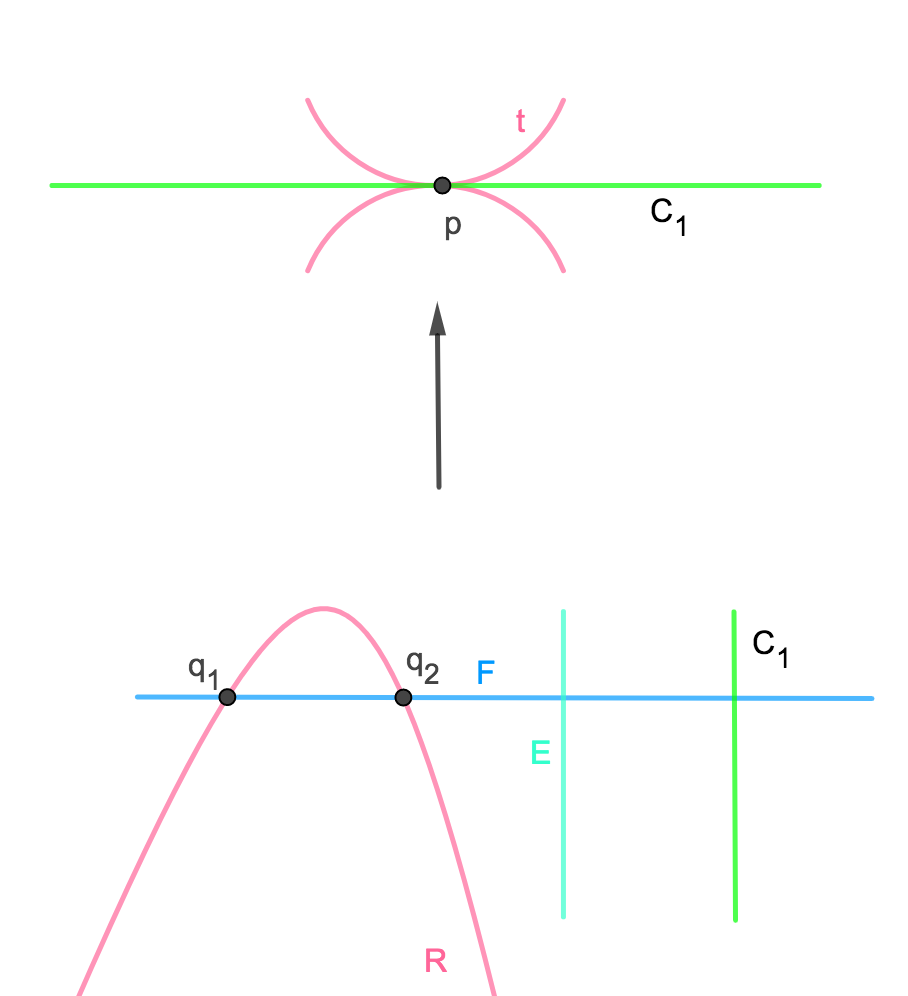}
\caption{The birational map $\sigma_3 \circ \sigma_4 \colon B' \rightarrow A$}\label{Fig4}
\end{figure}

In addition, one gathers  the following information:
$E'\cong \PP^1$ and $(E')^2=-2$, $F\cong \PP^1$ and $F^2=-1$,  $g(C_1)=1$  and $C_1^2=-2$. 

\item Second, we consider a double cover of $\beta\colon S'\longrightarrow B'$ ramified over $R+C_1+E'$ (that's even since $t+C_1$ is even on $A$). The surface $S'$ is a surface of general type, not minimal. Indeed, it contains a $-1$-curve, which is $\hat{E}=\beta^{-1}(E')$. The ramification divisor is denoted $\hat{R}+\hat{C_1}+\hat{E}$. Notice that $\hat{C_1}$ has genus $1$ and $\hat{C_1}^2=-1$.

\item Finally, to get $S$ we contract the $-1$-curve $\hat{E}$.
\end{enumerate}

We can summarize the construction of $S$ with the following diagram.
  \[
\begin{xy}
\xymatrix{
S'   \ar@{->}[d]    \ar@{->}[r]^{\beta}   &  B' \ar@{->}[d]^{\sigma_4} \\
S   \ar@{->}[dr]_{\alpha}  &  B \ar@{->}[d]^{\sigma_3} \\ 
& A 
 }
\end{xy}
\]

We note that since $\alpha$ is the Albanese morphism of $S$, we obtained in particular that the Albanese variety of these surfaces is isogenous to a product of elliptic curves:  
\begin{prop}\label{prop_isogenous}
The Albanese variety $A$ of the surface $S$ is isogenous to a product via an isogeny $\iota \colon A \rightarrow T_1 \times T_2$ of degree $2$.
\end{prop}

\section{Rito's family has three components of moduli dimension $2$}\label{sec_ellfib} 
The surfaces $S$ are constructed by a configuration of plane curves determined by two parameters (as noticed already in \cite[Section 3, Step 4]{rito}), that we denoted by $a,b$, and a choice of a linear system $|L|$ such that $|2L|$ contains the divisor $|C_1+t|$. So there are $2^4$ possible choice for $L$, since we can always add to $L$ a $2-$torsion line bundle.
In this section we prove that the family has three connected components, all irreducible  of moduli dimension $2$. 

\begin{defin}
Denote by $\mathcal{M}$ the locus of the surfaces $S$ above in the Gieseker moduli space of the surfaces of general type.
\end{defin}

The isogeny $\iota$ induces two natural fibrations  $f_i\colon A \longrightarrow T_i$ with fibres $\Lambda_i$, $i=1,2$ of genus $1$. The fibres of each fibration are isomorphic to the base of the other fibration: $i \neq j  \Rightarrow \Lambda_i \cong T_j$.

\begin{rem}\label{rem_fib1}
We label the the ramification points of $\phi_1 \colon T_1 \rightarrow \PP^1$ as $a_1,a_2,a_3,a_4$ using Figure \ref{Fig3}  as follows.

$\phi_1(a_1)$ be the projection of the line labeled $E_1$ 

$\phi_1(a_2)$ be the projection of the line labeled $E_2$ 

$\phi_1(a_3)$ be the projection of the line labeled $r_3$ 

$\phi_1(a_4)$ be the projection of the line labeled $r_4$ 

Similarly, we label the the ramification points of $\phi_2 \colon T_2 \rightarrow \PP^1$ as $b_1,b_2,b_3,b_4$ 
so that 

$\phi_2(b_1)$ be the projection of the line labeled $C_1$ 

$\phi_2(b_2)$ be the projection of the line labeled $C_2$ 

$\phi_2(b_3)$ be the projection of the line labeled $l$ 

$\phi_2(b_4)$ be the projection of the line labeled $E_0$ 
\end{rem}

Both fibrations have been considered in \cite[Section 3, Step 3]{rito}. 
The fibration $f_1$ is the pull back of the pencil of the lines through the point $p_0$ and $a_j$ corresponds to the line $r_j$. So, the branching points of $\phi_1$ correspond to the lines $r_1$, $r_2$, $r_3$ and $r_4$, that, in the natural coordinates, give the $4$ points $(1:0), (0:1),(1:1)$ and $(a:b)$, with cross-ratio $\frac{a}{b}$.

The fibration $f_2$ is given by the pencil of conics tangent to the lines $r_i$ in the points $p_i$, $i=1,2$: $b_1$ corresponding to the conic $C_1$, $b_2$ corresponding to $C_2$, $b_3$ corresponding to $2l$, $b_4$ corresponding to $r_1+r_2$.
Writing this pencil as $\langle x_0^2,x_1x_2 \rangle$ we get a parametrization of $\PP^1$ such that the branching points of $\phi_2$ have coordinates $(1:1), (1:ab),(1:0)$ and $(0:1)$, with cross-ratio $ab$.

We deduce the following 
\begin{prop}\label{prop_moduli_dim}
Every connected component of $\mathcal{M}$ has dimension $2$.
\end{prop}
\begin{proof}
The base of the family of the surfaces $S$ has a finite proper map on an open subset of $\CC^2$ given by the parameters $(a,b)$.
So, if $\mathcal{C}$ is any irreducible component of it,  $\dim \mathcal{C}=2$.

The relative Albanese morphism maps $\mathcal{C}$ to the moduli space of the Abelian surfaces with a polarization of type $(1,2)$. By Proposition \ref{prop_isogenous} the image of ${\mathcal C}$ is contained in the $2-$dimensional subvariety $\mathcal{I}$ of those isogenous to a product of curves. Since  these curves are double covers of $\PP^1$  branched at $4$ points with cross-ratio respectively $\frac{a}{b}$ and $ab$ the general pair of curves of genus $1$ appears in the image of $\mathcal{C}$:  the map $\mathcal{C}\rightarrow \mathcal{I}$ is generically finite and therefore dominant. 

Since isomorphic manifolds have isomorphic Albanese varieties, $\mathcal{C}\rightarrow \mathcal{I}$ factors through the moduli space of the surfaces of general type, and then the moduli dimension of $\mathcal{C}$  is $2$.
\end{proof}

We can now determine the isogeny. Recall that an \'etale double cover of a variety is determined up to isomorphism by a $2-$torsion line bundle on it, the antiinvariant part of the direct image of the structure sheaf of the cover.
Moreover for $\{i, j,h,k\} = \{1,2,3,4\}$, $\mathcal{O}_{T_1}(a_i-a_j) \cong \mathcal{O}_{T_1}(a_k-a_l)$ and $\mathcal{O}_{T_2}(b_i-b_j)\cong \mathcal{O}_{T_2}(b_k-b_l)$ are the $2-$torsion line bundles on the curves $T_1$, $T_2$.

\begin{lem}\label{lem_A}
The antiinvariant part of $\iota_*\oo_A$ is isomorphic to $\oo_{T_1}(a_4-a_3) \boxtimes \oo_{T_2}(b_2-b_1)$.
\end{lem}
\begin{proof}
By Remark \ref{rem_fib1} we can write every line bundle of torsion $2$ on $T_1 \times T_2$ as $\oo_{T_1}(a_i-a_j) \boxtimes \oo_{T_2}(b_k-b_l)$. 
We compute separately each factor by restricting to a fibre of type $\Lambda_1$ resp. $\Lambda_2$.  In fact, restricting the isogeny to a fibre of type $\Lambda_1$ (respectively $\Lambda_2$) we obtain an \'etale double cover of $T_2$ (respectively $T_1$) given by the restriction of the above bundle $\oo_{T_2}(b_k-b_l)$ (respectively $\oo_{T_1}(a_i-a_j)$).

We did the computation by using the fibres over $a_3$ and $b_1$. 

First consider the curve $\hat{r}_3:=f_1^{-1}(a_3) \subset A$, we need to show that the antiinvariant part of $(\iota_{|\hat{r}_3})_* \oo_{\hat{r}_3}$ is $\oo_{T_2}(b_2-b_1)$.

It  is invariant by the $(\ZZ/2\ZZ)^2$ action on $A$ given by the bidouble cover $\pi$, and in fact $\hat{r}_3$ lies in the locus of the fixed points of one of the three involutions. Thus $\pi$ induces a nontrivial involution on it, whose quotient is the double cover $\hat{r}_3 \rightarrow r_3$ branched on $p_3+ p_5+p_8+p_{10}$ (see Figure \ref{Fig1}). The involution $j$ acts on $r_3$ permuting those points as $p_3 \leftrightarrow p_5$, $p_8 \leftrightarrow p_{10}$ lifting to an involution on $\hat{r}_3$ without fixed points. Taking the quotient we get a commutative diagram

 \[
\begin{xy}
\xymatrix{
\hat{r}_3   \ar@{->}[d]^{\pi}    \ar@{->}[r]^{\iota}  &  T_2  \ar@{->}[d]\\
r_3   \ar@{->}[r]^{\sigma}  &  \PP^1. \\ 
 }
\end{xy}
\]
Let us call $q_j$ the ramification point in $A$ of $\pi_{|\hat{r}_3}$ mapping to $p_j$. Then $\iota(q_3)=\iota(q_5)=b_1$, $\iota(q_8)=\iota(q_{10})=b_2$. So $\iota^*\oo_{T_2}(b_2-b_1)=\oo_{\hat{r}_3}(q_8+q_{10}-q_3-q_5)\cong \oo_{\hat{r}_3}$: this implies the claim.

The analogous computation for the elliptic curve $C_1=f_2^{-1} (b_1)$ leads to consider the $4$ points on the corresponding conic  cut by the lines $r_3$ and $r_4$, permuted by $j$ as $p_3 \leftrightarrow p_5$, $p_7 \leftrightarrow p_9$. A fully analogous computation leads to $\iota^*\oo_{T_1}(a_4-a_3)\cong \oo_{C_1}$ completing the proof.
\end{proof}

Recalling that the kernel of
$
\iota^* \colon \Pic (T_1 \times T_2) \rightarrow \Pic A
$
is a subgroup of order $2$ generated by the antiinvariant part of  $\iota_* \OO_A$, we deduce
\begin{equation}\label{eqn_RedDiv}
f_1^*\bar{a} \otimes f_2^*\bar{b} \equiv f_1^*\bar{a}' \otimes f_2^*\bar{b}' \Leftrightarrow \text{ either } (\bar{a},\bar{b})=(\bar{a}',\bar{b}') \text{ or } \bar{a}+a_4 \equiv \bar{a}'+a_3 \text{ and } \bar{b}+b_2 \equiv\bar{b}'+b_1 
\end{equation}
that can be written equivalently as
\begin{equation}\label{eqn_RedDiv'}
f_1^*\bar{a} \otimes f_2^*\bar{b} \equiv f_1^*\bar{a}' \otimes f_2^*\bar{b}' \Leftrightarrow \text{ either } (\bar{a},\bar{b})=(\bar{a}',\bar{b}') \text{ or } \bar{a}+a_4 \equiv \bar{a}'+a_3 \text{ and } \bar{b}+b_3 \equiv\bar{b}'+b_4 
\end{equation}


\begin{prop}\label{prop_2L} 
\[
C_1+t \in  |f_1^*(a_4+a_3) +f_2^*(b_3+b_1)|
\]
\end{prop}
\begin{proof}
We compute
\[
t=\pi^*t\equiv \pi^* l =f_1^*a_1+f_1^*a_2+f_2^*b_3
\]
and the result follows since $C_1=f_2^*b_1$, $a_1+a_2 \equiv a_3+a_4$.
\end{proof}

It follows that we have the following description of the $16$ possible linear systems $L$.
\begin{prop}\label{prop_12pol}
$|L|$ varies among the linear systems $|f_1^*\bar{a} \otimes f_2^*\bar{b}|$ where $\bar{a}$ and $\bar{b}$ solve one of the following
\begin{itemize}
\item[1)] either $2\bar{a} \equiv 2a_3$ and $2\bar{b} \equiv b_4+b_1$
\item[2)] or $2\bar{a} \equiv a_4+a_3$ and $2\bar{b} \equiv b_3+b_1$
\end{itemize}
\end{prop}

Notice that each of the two systems of equations has $16$ distinct solutions $(\bar{a},\bar{b})$, divided in pairs by the equivalence relation 
\eqref{eqn_RedDiv}; so it gives $8$ distinct linear systems. We get then $16$ different possible choices of $|L|$ as expected. 

\begin{proof}
If $(\bar{a},\bar{b})$ solves the system 2), then $2(f_1^*\bar{a} + f_2^*\bar{b}) \equiv f_1^*(a_4+a_3) +f_2^*(b_3+b_1)\equiv C_1+t $ by Proposition \ref{prop_2L}.
If $(\bar{a},\bar{b})$ solves the system 1), then $2(f_1^*\bar{a} + f_2^*\bar{b}) \equiv f_1^*(2a_3) +f_2^*(b_4+b_1)\equiv f_1^*(a_4+a_3) +f_2^*(b_3+b_1)$ again by \eqref{eqn_RedDiv'}.

So all these linear systems are possible choices. Since they are $16$, they are all possible choices.
\end{proof}

We observe that $a_3 \in T_1$ and $b_1 \in T_2$ are the images of the essential singularity of the branching curve of the Albanese map of $S$.

Inspecting the linear equivalences in Proposition \ref{prop_12pol} we observe that in all cases $\OO_{T_2}(\bar{b}-b_1)$ is a $4-$ torsion line bundle. On the contrary $\OO_{T_1}(\bar{a}-a_3)$ is a torsion line bundle whose torsion order may change: it is $4$ in case 2) whereas in case 1) there are two possibilities: it may be $2$ or $1$. Recalling that we have two pairs  $(\bar{a},\bar{b})$ for each choice of $|L|$ (and then for each $S$) we deduce the following natural decomposition of $\mathcal{M}$. 

\begin{defin}
We say that a surface $S\in \mathcal{M}$ is of type $j$ if the minimal (among the two possible choices of $\bar{a}$) torsion order of  $\OO_{T_1}(\bar{a}-a_3)$  is $j$.

By Proposition \ref{prop_12pol} the values that $j$ assumes are $1,2,4$. 

Setting $\mathcal{M}_j$ for the subset of $\mathcal{M}$ of the surfaces of type $j$ we observe that each $\mathcal{M}_j$ is open and therefore we have decomposed
\[
\mathcal{M} = \mathcal{M}_1 \cup \mathcal{M}_2 \cup \mathcal{M}_4 
\]
as union of disjoint not empty open subsets.
\end{defin}

Now we prove that each $\mathcal{M}_j$ is irreducible.

 \begin{defin}\label{defin_Nj}
 We denote (as usual) by $\mathcal{M}_{1,3}$ the moduli space of the curves of genus $1$ with three ordered marked points. We are not assuming the points to be distinct.

We denote an element of $\mathcal{M}_{1,3}$  as $(C,x,y,z)$ where $C$ is a curve of genus $1$ and $x,y,z \in C$.
 
 We denote by $\mathcal{N}_j$, $j \in \{1,2,4\}$ the subvariety of $ \mathcal{M}_{1,3} \times \mathcal{M}_{1,3}$ of the form
 \[
 \left(
 \left(T_1, a_3,a_4,\bar{a} \right)
 ,
 \left(T_2, b_1,b_2,\bar{b} \right)
 \right)
 \]
 such that
 \begin{enumerate}
 \item $\oo_{T_1}(a_4-a_3)$ and $\oo_{T_2}(b_2-b_1)$ are torsion line bundles of torsion order $2$.
 \item  $\oo_{T_2}(\bar{b}-b_1)$ is a  torsion line bundle of torsion order $4$ such that $\oo_{T_2}(\bar{b}-b_1)^{\otimes 2} \not\cong \oo_{T_2}(b_2-b_1)$.
 \item
 \subitem if $j=1$: $\bar{a}=a_3$
 \subitem if $j=2$: $\bar{a}\neq a_4$ and $\oo_{T_1}(\bar{a}-a_3)$ is a torsion line bundle of torsion order $2$
 \subitem if $j=4$: $\oo_{T_1}(\bar{a}-a_3)^{\otimes 2} \cong \oo_{T_1}(a_4-a_3)$
 \end{enumerate}
 \end{defin}
 
 Note the correspondence among conditions $1,2,3$ and almost all solutions of the equations Proposition \ref{prop_12pol}. 
 The only solutions that do not have a counterpart here are those with $\bar{a}=a_4$. This is the reason for the map in the next statement to have a different degree for $j=1$.

\begin{prop}\label{prop_NM}
For each $j=1,2,4$ there is  a proper finite surjective morphism \[\mathfrak{m}_j \colon \mathcal{N}_j \rightarrow \mathcal{M}_j.\] Moreover $\mathfrak{m}_j$ is birational if $j=1$ whereas 
$\deg \mathfrak{m}_j=2$ if $j=2$ or $4$.
\end{prop}
 \begin{proof}
 We construct the map $\mathfrak{m}_j$.
 
 For every
$ \left(
 \left(T_1, a_3,a_4,\bar{a} \right)
 ,
 \left(T_2, b_1,b_2,\bar{b} \right)
 \right) \in \mathcal{N}_j$ we consider the isogeny $\iota \colon A \rightarrow T_1 \times T_2$ given by the $2$-torsion bundle $
\oo_{T_1}(a_4-a_3) \boxtimes \oo_{T_2}(b_2-b_1)$.

Now we construct a bidouble cover $\pi \colon A \dashrightarrow \PP^2$ as in Rito's construction.

We consider each $T_j$ with the group structure such that $a_3$ and $b_1$ are the respective neutral elements.
This fixes an action of the Klein group $K\cong (\ZZ/2\ZZ)^2$ as group automorphisms of $T_1 \times T_2$ by $(z_1,z_2)\mapsto (\pm z_1,\pm z_2)$.

Then we choose a point $p \in \iota^{-1}(a_3,b_1)$ and consider $A$ with the group structure such that $p$ is the neutral element, so that $\iota$ is a group homomorphism.
Considering the analogous action of the Klein group on $A$ we get a commutative diagram
 \[
\begin{xy}
\xymatrix{
A   \ar@{->}[d]^{/K}    \ar@{->}[r]^{\iota}  & T_1 \times T_2  \ar@{->}[d]^{/K}\\
D   \ar@{->}[r] &  \PP^1 \times  \PP^1. \\ 
 }
\end{xy}
\]

The bidouble cover $T_1 \times T_2 \rightarrow \PP^1 \times \PP^1$ is ramified at the union of $8$ elliptic curves, mapping to the four $2-$torsion points on each factor, including 
$a_3,a_4$ on $T_1$ and $b_1,b_2$ on $T_2$. We label the remaining points on $T_1$ as $a_1,a_2$ and the remaining points on $T_2$ as $b_3,b_4$. 
We note that the $2$-torsion bundle $\oo_{T_1}(a_4-a_3) \boxtimes \oo_{T_2}(b_2-b_1)$, when restricted to them, is not trivial. This implies that their preimage on $A$, ramification locus of the $8:1$ morphism $A \rightarrow \PP^1 \times \PP^1$ is again union of $8$ elliptic curves naturally labeled  as $a_1,\ldots,a_4,b_1,\ldots,b_4$.

A direct computation shows that the Klein group of $A$ acts on each of them, action that is faithful exactly on the curves labeled $a_1,a_2,b_3,b_4$. 
So the ramification locus of the double cover $D \rightarrow \PP^1 \times \PP^1$ is the image of them, union of $4$ rational curves, two on each ruling. Therefore $D$ is a Del Pezzo surface of degree $4$ with $4$ nodes. Solving the $4$ nodes we obtain a weak Del Pezzo surface with a configuration of $8$ rational curves whose incidence graph is an octagon with alternating self intersections $-1$ and $-2$: the strict transforms of the ramification lines have self intersection $-1$ whereas the exceptional curves have self intersection $-2$.

\begin{center}
\begin{tikzpicture}[line cap=round,line join=round,>=triangle 45,x=0.5cm,y=0.5cm]
\clip(-11.229778038427352,-4.747263293800263) rectangle (10.366603037338223,10.193225655392904);
\fill[line width=2pt, fill=white] (-2,-2) -- (2,-2) -- (4.82842712474619,0.8284271247461898) -- (4.82842712474619,4.828427124746189) -- (2,7.65685424949238) -- (-2,7.656854249492381) -- (-4.82842712474619,4.828427124746191) -- (-4.828427124746191,0.8284271247461912) -- cycle;
\draw [line width=2pt] (-2,-2)-- (2,-2);
\draw [line width=2pt] (2,-2)-- (4.82842712474619,0.8284271247461898);
\draw [line width=2pt] (4.82842712474619,0.8284271247461898)-- (4.82842712474619,4.828427124746189);
\draw [line width=2pt] (4.82842712474619,4.828427124746189)-- (2,7.65685424949238);
\draw [line width=2pt] (2,7.65685424949238)-- (-2,7.656854249492381);
\draw [line width=2pt] (-2,7.656854249492381)-- (-4.82842712474619,4.828427124746191);
\draw [line width=2pt] (-4.82842712474619,4.828427124746191)-- (-4.828427124746191,0.8284271247461912);
\draw [line width=2pt] (-4.828427124746191,0.8284271247461912)-- (-2,-2);
\draw [line width=2pt] (-5.967662863728063,3.6891913857643193)-- (-0.6889146077856187,8.967939641706764);
\draw [line width=2pt] (0.6273973422612953,9.029456907231085)-- (6.022481647921479,3.634372601570898);
\draw [line width=2pt] (3.6030683005844835,7.656854249492379)-- (-3.9006069062097364,7.656854249492383);
\draw [line width=2pt] (4.828427124746189,7.119013777409199)-- (4.828427124746189,-1.4026018644152431);
\draw [line width=2pt] (6,2)-- (0.8607642610181288,-3.1392357389818715);
\draw [line width=2pt] (-0.7768614643952754,-3.223138535604724)-- (-6.360991850846787,2.3609918508467875);
\draw [line width=2pt] (-4.82842712474619,6.420997479102759)-- (-4.828427124746192,-1.9551981005745072);
\draw [line width=2pt] (-3.5225147446270832,-1.955198100574507)-- (3.5739842881550468,-1.9842821130039419);
\draw (-0.9049536259779419,-0.41374544181445444) node[anchor=north west] {-1};
\draw (-5.87831975141132,3.163588087006045) node[anchor=north west] {-1};
\draw (-0.20693732767150286,8.689550448598686) node[anchor=north west] {-1};
\draw (5.231772996632834,3.19267209943548) node[anchor=north west] {-1};
\draw (-2.766330421461779,6.537333528820499) node[anchor=north west] {-2};
\draw (2.672379902842558,6.095501553679369) node[anchor=north west] {-2};
\draw (2.846883977419168,0.65785894349802905) node[anchor=north west] {-2};
\draw (-3.667934806774263,0.6623630180746388) node[anchor=north west] {-2};
\end{tikzpicture}
\end{center}

Now we consider, among the $-1$ curves in the octagon, the one `labeled'  $b_3$: contract first the other three $-1$ curves and then the two exceptional curves now of self intersection $-1$: the resulting surface is $\PP^2$ and the remaining three sides of the octagon map to three lines, let's call them $l$ (the one coming from the $-1$-curve ``'$b_3$''), $r_1$ and $r_2$.  The preimages of the lines of $\PP^1 \times \PP^1$ labeled $a_3,a_4,b_1,b_2$ are respectively  two lines $r_3$ and $r_4$ and two conics $C_1$, $C_2$ forming  the configuration of curves in Figure \ref{Fig1}.

Notice that the two points in $\iota^{-1}(a_3,b_1)$ map bijectively to the intersection points of $r_3$ and $C_1$. 
 We draw the tangent $t$ to $C_1$ in the image of $p$. Pulling-back $t$ and adding the elliptic curve dominating $C_1$ we obtain a divisor in $A$ as in Proposition \ref{prop_invariantiRito}. We have recovered Rito's construction.

Then Proposition \ref{prop_2L} applies and we have $16$ double covers $S\rightarrow A$ branched on this divisor, determined by the $32$ solutions of the equations in Proposition \ref{prop_12pol}. Since the pair $(\bar{a},\bar{b})$ is a solution by assumption, we can define as image of our element in $\mathcal{N}_j$ the surface $S$ of the corresponding double cover.
Then $S$ belongs to $\mathcal{M}_j$ by construction.

It is important to recall here that we have done an arbitrary choice in this construction, when we chose $p \in \iota^{-1}(a_3,b_1)$. We notice now that the isomorphism class of the surface $S$ does not depend on this choice, since the two corresponding double covers of $A$ are conjugated by the involution of $A$ given by the isogeny. So we have a well defined morphism $\mathcal{N}_j \rightarrow \mathcal{M}_j$.

Finally since there are two pairs of possible $(\bar{a},\bar{b})$ for each $|L|$ the maps $\mathfrak{m}_j$ are proper of degree $2$ for $j \geq 2$. For $j=1$ the degree is $1$ because we are not considering the solutions with $\bar{a}=a_4$. The surjectivity is obvious.
 \end{proof}
 
 

Now, we shall deal with the problem of irreducibility of $\mathcal{N}_j$ for $j=1,2$ and $4$, to do that we need to introduce some notation. 

Let us recall some well known fact about modular curves, see e.g. \cite[Section 1.5]{DS05}.  The principal congruence subgroup of level N is 
\[
\Gamma(N):=\Bigg\{  \begin{pmatrix} a & b \\ c & d \end{pmatrix} \in \mbox{SL}_2(\ZZ)| \begin{pmatrix} a & b \\ c & d \end{pmatrix} \equiv \begin{pmatrix} 1 & 0 \\ 0 & 1 \end{pmatrix} \mbox{ mod } N \Bigg\}.
\]
A subgroup $\Gamma$ of $\mbox{SL}_2(\ZZ)$ is
a congruence subgroup of level $N$ if $\Gamma(N) \subset \Gamma$ for some $N \in  \ZZ^+$. The
most important congruence subgroups are
\[
\Gamma_1(N):=\Bigg\{ \begin{pmatrix} a & b \\ c & d \end{pmatrix} \in \mbox{SL}_2(\ZZ)| \begin{pmatrix} a & b \\ c & d \end{pmatrix} \equiv \begin{pmatrix} 1 & * \\ 0 & 1 \end{pmatrix} \mbox{ mod } N \Bigg\}.
\]
The modular curve $\mathcal{Y}(\Gamma)$ for $\Gamma$ is defined as
\[
\mathcal{Y} (\Gamma) = \Gamma \setminus \mathcal{H} = \{\Gamma \cdot z | z \in  \mathcal{H}\}.
\]
and the special cases of modular curves for $\Gamma_1(N)$ 
denoted by $\mathcal{Y}_1[N] = \mathcal{H}/\Gamma_1(N)$.
\begin{thm} Points of $\mathcal{Y}_1[N]$ correspond to pairs $(E, P)$, where $E$ is an elliptic curve
and $P \in E$ is a point of exact order $N$. Two such pairs $(E, P)$ and $(E_0
, P_0)$ are identified when there is an isomorphism of $E$ onto $E_0$
taking $P$ to $P_0$. 
\end{thm}
 We are interested in the case when $N=4$ and in the special modular curve $\mathcal{Y}_1[4]$ which parametrizes elliptic curves with 4-torsion points. 
 
Now, let $\mathcal{Y}_1[2, \, 4]$ the space parametrizing an elliptic curves with a 2-line bundle point $\mQ$ and a 4-torsion line bundle $\mT$ such that $\mT^2 \neq \mQ$, than we have the following proposition. 

 \begin{prop} \label{prop.monodromy.2}
$\mathcal{Y}_{1}[2, \,4]$ is irreducible and generically
smooth of dimension $1$.
\end{prop}
\begin{proof}
We proceed as explained in the Appendix A of \cite{PP}.  Let $E=\CC/\Lambda$ be an elliptic curve (and $\widehat{E}$ its dual abelian variety), $E[n]$ the subgroup of order $n$ torsion points on $E$ and $\widehat{E}[n] \subset \widehat{E}$ the subgroup of $n$ torsion line bundles. Moreover let  $G=\mbox{SL}_2(\ZZ)$ be the modular group. Then $G$ is the orbifold fundamental group of $\mathcal{H}/G$ and there is an
induced monodromy action of $G$ on both $E[n]$ and 
$\widehat{E}[n]$, see \cite{Har79}.

By the Appell-Humbert theorem,
the elements of $\widehat{E}[2]$ can be canonically identified with the $4$ characters
$\Lambda \to \mathbb{C}^*$ with values in $\{\pm 1\}$  (see
\cite[Chapter 2]{BL04}) which are
\begin{equation*}
\chi_0:=(1,\,1), \quad \chi_1:=(1, \, -1), \quad
\chi_2:=(-1, \, 1), \quad \chi_3:=(-1, \, -1).
\end{equation*}

Let $\{\omega_1, \omega_2\}$ be a suitable basis of $\Lambda$ by \cite[proof of Proposition 8.1.3]{BL04}, the monodromy action of 
\[
M = \begin{pmatrix} \alpha & \beta \\ \gamma &\delta \end{pmatrix}\in G
\]
 induced over a character $\chi$ is as
follows:

\begin{equation} \label{eq.action.lattice}
\begin{split}
(M \cdot \chi)(\omega_1) & = \chi(\omega_1)^{\alpha}
\chi(\omega_2)^{\beta}  \\
(M \cdot \chi)(\omega_2) & =  \chi(\omega_1)^{\gamma}
\chi(\omega_2)^{\delta}.
\end{split}
\end{equation}
Therefore we have

\begin{equation}\label{eq_actionchi}
M \cdot \chi_1= ((-1)^{\beta}, \, (-1)^{\delta}),\quad
M \cdot \chi_2=((-1)^{\alpha}, \, (-1)^{\gamma}), \quad
M \cdot \chi_3=((-1)^{\alpha+\beta}, \, (-1)^{\gamma+\delta}),
\end{equation}

Whereas the $16$ elements of $\widehat{E}[4]$ correspond to he $16$ characters
$\Lambda \to \mathbb{C}^*$ with values in $\{\pm i\}$:
\begin{equation*}
\begin{array}{cccc}
 \psi_1 :=(1, \, 1),  &  \psi_2:=(1, \, -1), & \psi_3:=( -1,\, 1), & \psi_4 :=(-1, \, -1,)\\
 \psi_5 :=(1, \, i), & \psi_6:=(-1, \, i), &  \psi_7 :=(1, -i), & \psi_8:=(-1, \,  -i), \\
 \psi_9 :=(i, \, 1), & \psi_{10}:=(i, \, -1), & \psi_{11}:=(-i,  \, 1), & \psi_{12} :=(-i,  \, -1). \\
  \psi_{13} :=(i,\, i),& \psi_{14}:=(-i, \, i), & \psi_{15}:=(i, \, -i),  & \psi_{16} :=(-i, \, -i). \\
\end{array}
\end{equation*}

And by equations \eqref{eq.action.lattice} one can compute the induced  action of $M$ over a character $\psi$. 

Thus, to prove the first part of the proposition it is sufficient to check that the monodromy action of
$G$ is transitive on the set
\begin{equation*}
\{(\mQ, \, \mT) \, \in  \big(\widehat{E}[2] \setminus \mathcal{O}_E\big) \times  \big(\widehat{E}[4]\setminus  \widehat{E}[2]\big) | \,  \mT^2 \neq \mQ \}.
\end{equation*}
This is a straightforward computation which can be carried out as
the one in the proof of \cite[Proposition A1]{PP} and it is 
left to the reader.

Therefore we can consider the set of triples
\begin{equation*}
(z, \, \chi, \psi ), \quad z \in \mathcal{H}, \, \, \chi \in \{\chi_1,
\chi_2, \chi_3\} \subset \widehat{E_z}[2], \, \, \psi \in  \{ \psi_5, \ldots, \psi_{16} \} \subset
\widehat{E_z}[4].
\end{equation*}

The group $G$ acts on the set of triple $(z, \chi, \, \psi)$, with the natural action of the modular group on $\mathcal{H}$ and by the induced monodromy action 
 on the second two ones. The corresponding quotient
$\mathcal{Y}_{1}[2,4]$  is a quasi-projective variety.  Moreover 
\[
\pi \colon \mathcal{Y}_{1}[2,4] \longrightarrow \mathcal{H}/G
\] 
given by the forgetful map, is an 
$\acute{\textrm{e}}$tale covers on a smooth Zariski open set
$\mathcal{Y}_{1}^0 \subset \mathcal{H}/G$; then it is
generically smooth. Finally, by construction
$\mathcal{Y}_{1}[2,4]$ is a normal varieties, because it only has quotient singularities.
Then, since it is connected, it must be also irreducible.

\end{proof}

Finally, let $\mathcal{Y}_1[2, \, 2]$ the space parametrizing an elliptic curves with a 2-line bundle $\mQ$ and a second 2-torsion line bundle $\mT$ such that $\mT \neq \mQ$, than we have the following proposition.

  \begin{prop} \label{prop.monodromy.22}
$\mathcal{Y}_{1}[2, \, 2]$ is irreducible and generically
smooth of dimension $1$.
\end{prop}

\begin{proof} The proof is analogous to the one for Proposition \ref{prop.monodromy.2}. One has to be careful, again, in checking that the monodromy action of G is transitive on the set

\begin{equation*}
\{(\mQ, \, \mT) \, \in  \big(\widehat{E}[2] \setminus \mathcal{O}_E\big) \times  \big(\widehat{E}[2]\setminus  \mathcal{O}_E \big) | \,  \mT \neq \mQ \}.
\end{equation*}
But, again, this follows from the actions \eqref{eq_actionchi}, from which one sees right away that the image of $\mT$ is always different form the image of $\mQ$. 
\end{proof}

 \begin{prop}\label{prop_irred}
The subvariety  $\mathcal{N}_j \subset  \mathcal{M}_{1,3} \times \mathcal{M}_{1,3}$ is irreducible, generically smooth of
dimension $2$ for each $j=1,2$ and $4$. 
\end{prop}
\begin{proof} First of all, we mean by an {\it elliptic curve marked with a point} that we have fixed a group structure on the curve of genus $1$ for which that point is the neutral element. We always choose for $T_1$ the point $a_3$ and for $T_2$ the point $b_1$ as neutral elements. 

After this global consideration, we prove the claim case by case as $j$ varies. 
\medskip

{\bf Case  j=4:}

By Definition \ref{defin_Nj} the variety $\mathcal{N}_4$ depends
on the following data:
\begin{itemize}
\item one elliptic curve $T_1$ {\it marked} with a point $a_3$ and a $4-$torsion line bundle $\mathcal{T}_1=\oo_{T_1}(\bar{a}-a_3)$ which is not $2-$torsion -- its square determines the last point on $T_1$ ($a_4$);
\item one elliptic curve $T_2$ {\it marked} with a point $b_1$, a $4-$torsion line bundle $\mathcal{T}_2:=\oo_{T_2}(\bar{b}-b_1)$ and  a $2-$torsion 
line bundle $\oo_{T_2}(b_2-b_1) \not\cong \mathcal{T}^2_2$.
\end{itemize}
In other words there is a dominant morphism 
\[
\mathcal{Y}_1[4] \times \mathcal{Y}_1[2, \, 4] \rightarrow \mathcal{N}_4.
\] 
We observe that
$\mathcal{Y}_1[4]$ is a generically smooth quasi-projective variety,  connected,
 and irreducible of dimension  $1$, \cite[Chapter 2]{DS05}. By Proposition \ref{prop.monodromy.2}  $\mathcal{Y}_1[2, \, 4]$ is irreducible and generically smooth of dimension $1$. 
 This concludes the proof since $\dim \mathcal{N}_4=2$ by Proposition \ref{prop_moduli_dim} and Proposition \ref{prop_NM}.
 
\medskip
 
{\bf Case  j=2:}

By Definition \ref{defin_Nj} the variety $\mathcal{N}_2$ depends
on the following data:
\begin{itemize}
\item one elliptic curve $T_1$ {\it marked} with a point $a_3$ and two $2-$torsion line bundles $\mathcal{T}_1=\oo_{T_1}(a_4-a_3)$ and $\mathcal{Q}=\oo_{T_1}(\bar{a}-a_3)$ such that $\mathcal{T}_1 \not\cong \mathcal{Q}$;
\item one elliptic curve $T_2$ {\it marked} with a point $b_1$, a $4-$torsion line bundle $\mathcal{T}_2:=\oo_{T_2}(\bar{b}-b_1)$ and  a $2-$torsion 
line bundle $\oo_{T_2}(b_1-b_2) \not\cong \mathcal{T}^2_2$.
\end{itemize}
In other words there is a dominant morphism 
\[
\mathcal{Y}_1[2, \, 2] \times \mathcal{Y}_1[2, \, 4] \rightarrow \mathcal{N}_2.
\] 
By Proposition \ref{prop.monodromy.22}, we have that $\mathcal{Y}_1[2, \, 2]$ is irreducible and generically smooth of dimension $1$. 
 This concludes the proof since $\dim \mathcal{N}_2=2$ by Proposition \ref{prop_moduli_dim} and Proposition \ref{prop_NM}.

\medskip

{\bf Case  j=1:}

Finally, we have by Definition \ref{defin_Nj} that the variety $\mathcal{N}_1$ depends
on the following data:
\begin{itemize}
\item one elliptic curve $T_1$ {\it marked} with a point $a_3$ and one $2-$torsion line bundles $\mathcal{T}_1=\oo_{T_1}(a_4-a_3)$;
\item one elliptic curve $T_2$ {\it marked} with a point $b_1$, a $4-$torsion line bundle $\mathcal{T}_2:=\oo_{T_2}(\bar{b}-b_1)$ and  a $2-$torsion 
line bundle $\oo_{T_2}(b_1-b_2) \not\cong \mathcal{T}^2_2$.
\end{itemize}
In other words there is a dominant morphism 
\[
\mathcal{Y}_1[2] \times \mathcal{Y}_1[2, \, 4] \rightarrow \mathcal{N}_1.
\] 

We observe that
$\mathcal{Y}_1[2]$ is a generically smooth quasi-projective variety,  connected,
 and irreducible of dimension  $1$, \cite[Chapter 2]{DS05}. Ad we conclude as the previous cases. 
 
\end{proof}

\begin{cor}\label{prop_irred}
The components $\mathcal{M}_1$, $\mathcal{M}_2$ and  $\mathcal{M}_4$ of  $\mathcal{M}$ are irreducible of dimension $2$.
\end{cor}
\begin{proof}
By Proposition \ref{prop_NM} we have that $\mathfrak{m}_j \colon \mathcal{N}_j \rightarrow \mathcal{M}_j$ is a proper finte surjective morphism for each $j=1,2$ and $4$. Moreover, by Proposition \ref{prop_irred} we have that each $\mathcal{N}_j$ is irreducible of dimension $2$ for each $j=1,2$ and $4$.

\end{proof}

\section{Some remarks on the deformations  of a blown up surface}\label{sec_deform}
In this section we shall present some classicall results on deformation of a pairs. The main result is  Theorem \ref{thm_def}, possibly known to the experts, although we could not find it in the literature. This section will be employed systematically in the Moduli Space Section \ref{moduli} and Theorem \ref{thm_def} mainly for the Remark \ref{rem_InvEAntiInv}.

Let us first recall some basic definition.

Let $B$ an algebraic nonsingular variety over an algebraically closed field $k$. The \emph{first order deformation} of $B$ is a  commutative diagram
\[
\begin{xy}
\xymatrix{
B \ar@{->}[r] \ar@{->}[d] & \mathcal{B}  \ar@{->}[d]^{\pi} \\
Spec(k) \ar@{->}[r] & Spec(k[\epsilon]) 
 }
\end{xy}
\]
where $\pi$ is a flat morphism,  $Spec(k[\epsilon]) =Spec(k[t]/t^2)$ and such that the induced morphism 
\[
B \to Spec(k) \times_{Spec(k[\epsilon])}\mathcal{B}
\]
is an isomorphism. There is a natural notion of isomorphism between first order deformations, see \cite[Section 1.2]{Se06}. The set of first order deformations, up to isomorphisms, is usually denoted by $T^1(B)$ and it has a natural structure of complex vector space (see \cite{S68}). If $B$ has a semiuniversal deformation $\tilde{B} \to Def(B)$ then every first order deformation is induced by a unique map $Spec(k[\epsilon]) \to Def(B)$ and then there exists an isomorphisms of vector spaces
\[
T_0Def_{B} \cong T^1(B) \cong H^1(B, T_B),
\]
for the last isomorphism see e.g. \cite[Proposition 1.2.9]{Se06}.

Now, we look at deformations of subvarieties in a given variety. 
Given a closed embedding $D \subset B$, the \emph{first order deformation} of $D$ in $B$ is a cartesian diagram 
\[
\begin{xy}
\xymatrix{
D \ar@{^{(}->}[r] \ar@{->}[d] & \mathcal{D}  \ar@{->}[d]^{\pi}  \ar@{^{(}->}[r] & B \times Spec(k[\epsilon]) \ar@{->}[d] \\
Spec(k) \ar@{^{(}->}[r] & Spec(k[\epsilon]) \ar@{=}[r]&Spec(k[\epsilon]) 
 }
\end{xy}
\]
where $\pi$ is flat and it is induced by the projection from $B \times Spec(k[\epsilon])$. Again we can give a cohomological interpretation to these deformations, indeed there is a natural identification between the first order deformations of $D$ in $B$ and $H^0(D, \mathcal{N}_{D/B})$, where $\mathcal{N}_{D/B}$ is the normal sheaf of $D$ in $B$, see e.g. \cite[Proposition 3.2.1]{Se06}. 

Before introducing the last two situations we are interested in, let us recall the following definition. 
\begin{defin} Let $D_1, \ldots ,D_k$ be divisors in a smooth manifold X and $x_1, \ldots,x_k$ equations for them.
Define $\Omega^1_S(\log D_1, \ldots ,\log D_k)$ to be the subsheaf (as $\mathcal{O}_X$-module) of $\Omega^1_X(D_1+ \ldots +D_k)$ generated by $\Omega^1_X$ and by $\frac{dxj}{x_j}$ for $j=1, \ldots k$.
\end{defin}

The next situation we want to look at is the case of deformation of a pair $(B,D)$ where $j\colon D \hookrightarrow B$ is a closed embedding.  The deformation theory of morphisms is more subtle if we want to allow both the domain and  the target to deform nontrivially. A first order deformation of the pair $(D,B)$ is a commutative diagram
\[
\begin{xy}
\xymatrix{
\mathcal{D} \ar@{->}[rr]^{J} \ar@{->}[dr]_{\pi_D} & & \mathcal{B}  \ar@{->}[dl]^{\pi_B} \\
 & Spec(k[\epsilon]) 
 }
\end{xy}
\]
where $\pi_D$ and $\pi_B$ come from  first deformations of $D$ and $B$ respectively and $J$ is a closed embedding. There is a natural notion of isomorphism between first order deformations of pairs see e.g. \cite[Section 3.4]{Se06}. And, we denote by $Def'_j$ the set if isomorphism classes of first order deformations of the pair $(B,D)$, which are locally trivial. Also in this case we have a cohomological interpretation, by \cite[Proposition 3.4.17]{Se06},  $Def'_j$ has a formal semiuniversal deformation and its tangent space is isomorphic to $ H^1( T_{B'}( - \log D'))$, where $T_{B'}( - \log D')$ is the sheaf of germs of tangent vectors to $B'$ which are tangent to $D'$.

Finally, let us consider the following situation. Let $B$ be a compact complex smooth surface, $p \in B$ and $\sigma\colon B' \to B$ the blow up of $B$ in $p$ with exceptional divisor $E$. Let $D$ be an effective divisor on $B$ which has multiplicity $c$ in $p$. Moreover, let us denote by $D'=\sigma^*(D)-cE$ the strict transform of $D$ in $B'$ and assume that $D'$ is a smooth normal crossing divisor. We want to describe the relations between the deformations of the pair  $(B', \, D')$ with those of $D$ in $B$.

We know that the first order deformations of the pair $(B', \, D')$ are parameterized by the vector space $H^1(T_{B'}( - \log D'))$.
The natural map
\[
\vartheta \colon H^1( T_{B'}( - \log D'))  \to H^1(T_{B'}) 
\]
corresponds to the forgetful map, which forget the deformation of $D'$. 
By \cite[Exercise 10.5]{H} we have an exact sequence
\begin{equation*}
0 \to \sigma_*T_{B'} \to T_B \to T_pB \to 0
\end{equation*}
where $T_pB \cong \CC^2$ is the tangent space of $B$ in $p$ seen as skyscreaper sheaf concentrated in $p$. Then we consider the long exact sequence in cohomology and in particular the connecting homomorphism
\[
\psi\colon T_pB \to H^1(\sigma_* T_{B'}) \cong H^1(T_{B'}).
\]
The next result give us a better understanding of the intersection between the images of the maps $\vartheta$ and $\psi$ in $H^1(T_{B'})$.

\begin{thm}\label{thm_def} Keeping the same notation as before, assume that $D$ is smooth at $p$, so $c=1$, and choose an element $v \in T_pB$.

Then $\psi(v) $ is contained  in $Im(\vartheta)$ if and only if the class of $v$ in the normal vector space $T_pB/T_pD$ extends to a global section of the normal bundle $v_D \in H^0(D,{\mathcal N}_{D|B})$. 

In particular  $v$ is  tangent to $D$ if and only if $v_D$ vanishes in $p$.    
\end{thm}

\begin{proof} 
We start constructing a family of first order deformations of $B'$.

Let $U$ be an affine chart of $B$ centered in $p$ with local coordinates $x,y$ such that $D=\{x=0\}$. We consider a section $s_{a,b}$ of the trivial family $B \times Spec(\CC[\epsilon]) \rightarrow  Spec(\CC[\epsilon])$ whose image is contained in $U \times Spec(\CC[\epsilon])$
  \[
\begin{xy}
\xymatrix{
B \times Spec(\CC[\epsilon]) \ar@{}[r]|-*{\supset}& U \times Spec(\CC[\epsilon])  \ar@{->}[r]  
 & \ar@/{}^{2pc}/[l]^{s_{a,b}}   Spec(\CC[\epsilon]) 
 }
\end{xy}
\]
obtained by mapping $(x,y,\epsilon)$ to $(a\epsilon, b\epsilon, \epsilon)$, so that the image is locally the complete intersection
\[
x-a\epsilon=y-b\epsilon=0.
\]
Blowing up this section we obtain the following families over $Spec(\CC[\epsilon])$
\[
\begin{xy}
\xymatrix{
\mathcal{B}'_{a,b}  \ar@{->}[r]\ar@/{}^{2pc}/[rr]^{\Phi_{a,b}}  & B \times Spec(\CC[\epsilon])  \ar@{->}[r]  
 &    Spec(\CC[\epsilon]) \\
U'_{a,b}  \ar@{^{(}->}[u] \ar@{->}[r] & U \times Spec(\CC[\epsilon])   \ar@{^{(}->}[u]  \ar@{->}[r]  
 &   \ar@{=}[u]   Spec(\CC[\epsilon]) 
 }
\end{xy}
\]
where $\Phi_{a,b}$ is a first-order deformation of $B'$. The Kodaira-Spencer correspondence associates to $\Phi_{a,b}$ a class in $ \kappa\left(\Phi_{a,b} \right) \in H^1(B',T_{B'})$, its Kodaira-Spencer class. 
This can be explicitly computed: following {\it e.g.} the proof of \cite[Proposition 1.2.9]{Se06} we find
\[
\kappa \left( \Phi_{a,b} \right)=\psi \left( a\frac{\partial}{\partial x} + b\frac{\partial}{\partial y}\right).
\]

The blown up chart $U'_{a,b}$ is the subscheme of $U \times \PP^1 \times  B \times Spec(\CC[\epsilon])$ defined by
\[
Y(x-a\epsilon)=X(y-b\epsilon),
\]
where $(X,Y)$ are homogeneous coordinates on the factor $\PP^1$. It is the union of two affine charts, given respectively by imposing $X \neq 0$ and $Y\neq 0$.

Let us work locally and restrict to the affine chart of $U'_{a,b}$ given by $Y\neq 0$, and let us introduce the new coordinate $z=\frac{X}{Y}$. 
Then, we can eliminate $x$ by 
\[
x=zy+(a-bz)\epsilon
\]
and the exceptional divisor $\mathcal{E}$ of the blow-up is $\{y-b\epsilon=0\}$ in the coordinates $y,z$.

Since $D=\{x=0\}$, the strict transform of $D$ on $B'$ is, in the coordinates $y,z$, the divisor $D'=\{z=0\}$. Now, $\kappa \left( \Phi_{a,b}\right)$ is in the image of $\vartheta$ if and only if
$D'$ can be extended to a divisor $\mathcal{D}'_{a,b}$ in $\mathcal{B}'_{a,b}$. The  image of $\mathcal{D}'_{a,b}$ in  $B \times Spec(\CC[\epsilon])$ is 
\[
\mathcal{D}_{a,b}=\{x+\delta(x,y) \epsilon =0\},
\]
an infinitesimal deformation of $D$ in $B$ over $Spec \left( {\mathbb C}[\epsilon]\right)$ so that $\delta(x,y)$ is the affine trace of a global section of the normal bundle ${\mathcal N}_{D|B}$, an element $\delta \in H^0(D,{\mathcal N}_{D|B})$  (\cite[Proposition 3.2.1]{Se06}), locally given by the class of a vector field $
\delta(x,y)\frac{\partial}{\partial x}$.

The pullback of $\mathcal{D}_{a,b}$ on $\mathcal{B}'_{a,b}$ contains the exceptional divisor $\mathcal{E}$, thus
\[
y-b \epsilon  \text{ divides } zy+(a-bz-\delta(zy,y) )\epsilon =z(y-b\epsilon) + (a-\delta(zy,y) )\epsilon,
\]
that implies $\delta(0,0)=a$. Conversely, if $\delta(0,0)=a$ the pull-back of $\mathcal{D}_{a,b}$ contains $\mathcal{E}$ and then its strict transform gives an extension ${\mathcal D}'_{a,b}$ of $D'$ in $\mathcal{B}'_{a,b}$. 

Since the class of $v= a\frac{\partial}{\partial x} + b\frac{\partial}{\partial y}$ in $T_pB/T_pD$ equals the class of $ a\frac{\partial}{\partial x}$, then $\psi \left(v \right)$ is in the image of $\theta$ if and only if there is some $\delta \in H^0(B,{\mathcal N}_{D|B})$ whose value at $p$ is the class of $v$.
\end{proof}

The situation is even simpler if $D$ is a rigid divisor.

\begin{cor}\label{cor_def}
Let $D$ be a divisor which is smooth in $p$ and $H^0(D,\oo_D(D))=0$. Let  $v \in T_pB$ such that  $\psi(v) \in Im (\vartheta)$. Then $v$ is tangent to $D$. 
\end{cor}

Keeping the same notation as above, the application we have in mind for the next can be summarized in 
\begin{prop}\label{prop_deformation}  Let $B' \to B$ the blow up of $B$ in $p$ and $D'$ the strict transform of $D$ a divisor passing through $p$.  Let us further suppose that $D \geq D_1 + D_2$ with $D_1$ and $D_2$  smooth and transversal in $p$. Moreover, let us assume that $H^0(D_i,\oo_{D_i}(D_i))=0$ for $i=1,2$. Then 
\begin{equation}
\vartheta( H^1(B', \, T_{B'}( - \log D'))) \cap \psi(T_pB) = \{ 0 \} .
\end{equation}
\end{prop}

\begin{proof} 
We have that $\vartheta$ factors through the analogous map for $D_j$, $j=1,2$:
\[
H^1(T_{B'}(-\log(D'))) \to H^1(T_{B'}(-\log(D_j))) \to H^1(B',T_{B'}) \qquad \textrm{ for } j=1,2.
\]
Hence, the image of $\vartheta$ is contained in the image of both $H^1(T_{B'}(-\log(D_j)))$ for $j=1,2$. Than we apply the Corollary \ref{cor_def} and we obtain a vector $v$ which is tangent to both $D_1$ and $D_2$.
Finally, observe that if a vector is tangent to two transversal curves must vanish. 
\end{proof}

\begin{cor}\label{cor_deformation} Let $D$ be as in Proposition \ref{prop_deformation} and suppose moreover $H^0(D', \oo_{D'}(D'))=0$. Then the  composition
\[
 H^1(B', T_{B'}(-\log D')) \to H^1(B,T_{B})
\]
is injective.
\end{cor}
\begin{proof}
The proof follows directly from the Proposition \ref{prop_deformation} and the following diagram with exact row and column.
{\small
\[
\begin{xy}
\xymatrix{
&   H^0(\oo_{D'}(D')) = 0   \ar@{->}[d] \\
&  H^1(T_{B'}(-\log D')) \ar@{->}[d]^{\vartheta}    \ar@{->}[dr] \\
T_pB \ar@{->}[r]^{\psi}  & H^1(T_{B'}) \ar@{->}[r] & H^1(T_B) .
 }
\end{xy}
\]}
\end{proof}

We conclude the section with the following general result.

\begin{lem}\label{lem_prodElliptic} Let $A$ be an abelian surface isogenous to a product of elliptic curves $T_1 \times T_2$. 
Let $H \subset H^1(T_A)$ be the linear subspace corresponding to the projective deformations of $A$ and let $H_j \subset H^1(T_A)$ be the linear subspaces corresponding to the deformations preserving the fibration $A \longrightarrow T_j$ for $j=1,2$. Then $H$, $H_1$ and $H_2$ are three different hyperplanes such that the intersection of any two of them is contained in the third.
\end{lem}

\begin{proof}
The isogeny maps  $H^1(T_A)$ isomorphically to $H^1(T_{T_1\times T_2})$ by a map preserving $H$, $H_1$ and $H_2$. Therefore we may assume without loss of generality $A=T_1 \times T_2$.

 For a product of curves the period matrix assumes the form
\[
 \Lambda\,=\,\Omega\,\ZZ^4,\quad
\Omega\,:\,\ZZ^4\,\longrightarrow\,\CC^2,\qquad
x\,\longmapsto\, \Omega\,x=  \left( \begin{array}{c|c} \Delta_1 & \tau\end{array} \right)=
\left( \begin{array}{cc|cc} 1 & 0 & \alpha & 0 \\ 0 & 1 & 0 & \delta \end{array} \right)
\]

It is well known that one can identify the deformation space $H^1(T_A)$ of a polarized abelian surface  $A = V/\Lambda$  with the space of the square matrices $\tau$ (see  \cite[Chapter 1]{HKW}). For $\tau=\begin{pmatrix}
a&b\\
c&d\\
\end{pmatrix}$ we obtain the deformation given by 
\[
 \left( \begin{array}{c|c} \Delta_1 & \tau\end{array} \right)=
\left( \begin{array}{cc|cc} 1 & 0 & \alpha+a\epsilon & b\epsilon \\  0 & 1 & c\epsilon & \delta+d\epsilon \end{array} \right).
\]

The Riemann--Conditions for an abelian surface with a principal polarization yields the existence of an integral basis $\{\lambda_i\}_{i}$ for $\Lambda$ and a complex basis $\{e_i\}$ for $V$ such that the  period matrix can be normalized so that the matrix $\tau$  is symmetric with positive imaginary part (see \cite{GH} p.306), so
\[H=\{b=c\}.\]
The subspaces $H_j$ are respectively
\begin{align*}
b&=0&
c&=0
\end{align*}
and this concludes the lemma.
\end{proof}

\section{The moduli space}\label{moduli}

The following result can be found in \cite[Section 5]{Ca11}.
\begin{prop} \label{prop.degree.alb} Let $S$ be a minimal surface of
general type with $q(S) \geq 2$
 and Albanese map $\alpha \colon S \to A$, and assume that $\alpha(S)$
is a surface. Then this is a topological property. If in addition
$q(S)=2$, then the degree of $\alpha$ is a topological
invariant.
\end{prop}
\begin{proof}
By \cite{Ca91} the Albanese map $\alpha$ induces a homomorphism of
cohomology algebras
\begin{equation*}\alpha^*\colon H^*(\textrm{Alb}(S), \, \mathbb{Z}) \longrightarrow H^*(S, \, \mathbb{Z})
\end{equation*}
and $H^*(\textrm{Alb}(S), \mathbb{Z})$ is isomorphic to the full
exterior algebra 
\[ \bigwedge^*  H^1(\textrm{Alb}(S), \,
\mathbb{Z}))\cong\bigwedge^*  H^1(S, \, \mathbb{Z}).
\]
 In particular,
if $q=2$ the degree of the Albanese map equals the index of the
image of $\bigwedge^4  H^1(S, \, \mathbb{Z})$ inside $H^4(S, \,
\mathbb{Z})$ and it is therefore a topological invariant.
\end{proof}

Consider a surface $S$ in ${\mathcal M}$. 
By Proposition \ref{prop.degree.alb} it follows that one may study
the deformations of $S$ by relating them to those of the flat double
cover $\beta \colon S' \to B'$. By \cite[p. 162]{Se06} we have an
exact sequence
\begin{equation} \label{suc.def.S}
0 \longrightarrow T_{S'} \longrightarrow \beta^{*}T_{B'} \longrightarrow \mathcal{N}_{\beta}  \longrightarrow
0,
\end{equation}
where $\mathcal{N}_{\beta}$ is a coherent sheaf supported on the ramification divisor $\hat{R}+\hat{C_1}+\hat{E}$
called the \emph{normal sheaf of} $\beta$.

 \begin{lem}\label{Lem_Normal} Keeping the notation above it holds
 \begin{equation}\label{eq_normalS}
 H^i(S',\mathcal{N}_{\beta})=  H^i(\mathcal{O}_{\hat{R}}(2\hat{R})) \oplus  H^i(\mathcal{O}_{\hat{C_1}}(2\hat{C_1})) \oplus  H^i(\mathcal{O}_{\hat{E}}(2\hat{E})), \, i=0,1.
 \end{equation}
 Moreover we have:
 \begin{equation*}
 h^0\big(\mathcal{O}_{\hat{R}}(2\hat{R})\big)=0, \quad   h^0\big(\mathcal{O}_{\hat{C_1}}(2\hat{C_1})\big)=0, \quad  h^0\big(\mathcal{O}_{\hat{E}}(2\hat{E})\big)=0, 
 \end{equation*}
 \begin{equation*}
 h^1\big(\mathcal{O}_{\hat{R}}(2\hat{R})\big)=2,\quad    h^1\big(\mathcal{O}_{\hat{C_1}}(2\hat{C_1})\big)=2, \quad    h^1\big(\mathcal{O}_{\hat{E}}(2\hat{E})\big)=1. 
 \end{equation*}
 \end{lem}
  \begin{proof} The ramification divisor of the double cover $\beta \colon S' \longrightarrow B$ is the disjoint union of the divisors $\hat{E}$, $\hat{R}$ and $\hat{C_1}$, this is enough for \eqref{eq_normalS}. 
  
  Since $\hat{C}$ is an elliptic curve  with $\hat{C}^2=-1$, we have that $2\hat{C}$ is not effective on $\hat{C}$ and by Riemann--Roch we conclude that $h^1\big(\mathcal{O}_{\hat{C_1}}(2\hat{C_1})\big)=2$. 
  
  The computations for $\hat{E} \cong \PP^1$ are straightforward. 
  
  Finally we work on $\hat{R}$. Recall that $g(\hat{R})=3$ and $\hat{R}^2=0$. Thus, by Riemann--Roch we have $\chi\big(\mathcal{O}_{\hat{R}}(2\hat{R})\big)=-2$. Therefore, it is sufficient to prove that $h^0\big(\mathcal{O}_{\hat{R}}(2\hat{R})\big)=0$.
  
We notice that $H^0(\oo_{\hat{R}}(2\hat{R}))=H^0(\oo_R(R))=H^0(\mathcal{N}_{R|B})$. Recall that by adjunction the normal bundle of a curve in an abelian surface equals its canonical bundle, so $\mathcal{N}_{t|A}=\omega_t$. The map $\nu =(\sigma_4 \circ \sigma_3)|_R\colon R \longrightarrow t$ is the normalization of $t$. Let $q_1, q_2 \in R$ such that $\nu(q_i)=p$ with $i=1,2$ and recall that $p$ is the tacnode of $t$. We have
\[ \omega_R = \nu^*\omega_t \otimes \oo_R(-2q_1-2q_2), \quad \mathcal{\mathcal{N}}_{R|B} = \nu^*\mathcal{N}_{t|A}\otimes \oo_R(-4q_1-4q_2),
\]
 this yields
 \begin{equation}\label{eq_NRB}
 \mathcal{N}_{R|B} = \omega_R \otimes \oo_R(-2q_1-2q_2).
 \end{equation}

 By construction $R$ is a smooth irreducible curve of genus $3$ with a $(\ZZ/2\ZZ)^2$-action, by \cite[Lemma 2.15]{ortega} $R$ is not hyperelliptic. Thus, $R$ is a plane quartic curve invariant under the action 
 \[ (x_0 : x_1 : x_2) \mapsto (\pm x_0, \pm x_1, \pm x_2),
 \]
 the equation defining it is biquadratic, and the divisor $q_1 + q_2$ is invariant. 
 This means that $q_1$ and $q_2$ have a stabilizer of order $2$ and lie on a coordinate line $x_j$. 
 
 By \eqref{eq_NRB} 
  \[H^0(\oo_{\hat{R}}(2\hat{R})) = 0 \Leftrightarrow (x_j) \textrm{ is not a bitangent}
 \]

Since the quartic equation defining $R$ is biquadratic, this would imply that $R$ is singular in $q_1$ and $q_2$, but this is absurd. 
  \end{proof}
  
  Recall that $S'$ is a surfaces of general type, hence $h^0(T_{S'})=0$ and  using the bit of information of the previous lemma, the sequence \eqref{suc.def.S} induces the following long sequence in cohomology. 

\begin{equation*} \label{suc.coh.def.S}
0  \longrightarrow
H^1(T_{S'}) \longrightarrow H^1(\beta^{*}T_{B'}) \longrightarrow H^1(\mathcal{N}_{\beta})  \longrightarrow H^2(T_{S'}) \longrightarrow H^2(\beta^{*}T_B) \longrightarrow
0.
\end{equation*}

  \begin{prop} \label{prop.coh.TB}
Keeping the notation as above,  then the sheaf $\beta^*T_{B'}$ satisfies
\begin{equation*}\begin{split}
h^0(S',\beta^*T_{B'})=&h^0(B', T_{B'} \otimes \mathcal{L}_{B'}^{-1}),  \\  h^1(S',\beta^*T_{B'})=&6+h^1(B', T_{B'} \otimes \mathcal{L}_{B'}^{-1}), \\  h^2(S',
\,\beta^*T_{B'})=&2+h^2(B', T_{B'} \otimes \mathcal{L}_{B'}^{-1}). 
\end{split}
\end{equation*}
\end{prop}
\begin{proof}
Since $\beta \colon S' \to B'$ is a finite map, by using projection
formula and the Leray spectral sequence we deduce
\begin{equation*}
h^i(S', \, \oo_{S'})=h^i(B', \,
\oo_{B'})+h^i(B', \,  \mathcal{L}_{B'}^{-1}), \quad i=0,\,1,\,2.
\end{equation*}
Recall that $p_g(S')=q(S')=2$ and $B'$ is an abelian surface blown up twice, then we have
\begin{equation}\label{eq.LB} 
h^0(B', \, \mL_{B'}^{-1})=0, \quad h^1(B,\, \mL_{B'}^{-1})=0, \quad h^2(B', \,
\mL_{B'}^{-1})=1,
\end{equation}
By the same argument above we have
\begin{equation*} 
h^i(S', \, \beta^*T_{B'})=h^i(B', \, \beta_* \beta^* T_{B'})=h^i(B', \,
T_{B'})+h^i(B', \, T_{B'} \otimes \mathcal{L}_{B'}^{-1}), \quad i=0,\,1,\,2.
\end{equation*}
We look first at $\sigma_3$. There is a short exact sequence
\begin{equation} \label{eq.blow-up.tangent}
0 \lr T_{B} \to \sigma_3^*T_A \to \oo_E(-E) \to 0,
\end{equation}
see \cite[p. 73]{Se06} for the general setting of a blow up. Then a direct computation shows
\begin{equation*} 
h^0(B, \, T_{B})=0, \quad h^1(B, \, T_{B})=4, \quad h^2(B, \, T_{B})=2.
\end{equation*}
The analogous computation for $\sigma_4$, for the exact sequence
\begin{equation} \label{eq.blow-up.tangent4}
0 \lr T_{B'} \to \sigma_4^*T_{B} \to \oo_F(-F) \to 0.
\end{equation}
 yields
\begin{equation*} 
h^0(B', \, T_{B'})=0, \quad h^1(B', \, T_{B'})=6, \quad h^2(B', \, T_{B'})=2.
\end{equation*}
Therefore the claim follows.

\end{proof}
Let us consider the exact sequence
\begin{equation}\label{suc_fsupEF}
0 \lr T_{B'} \to (\sigma_4 \circ \sigma_3)^*T_{A} \to \mathcal{N}_{\sigma_4 \circ \sigma_3} \to 0, 
\end{equation}
where the last sheaf is supported on $E$ and $F$. We tensor \eqref{suc_fsupEF} by $\mathcal{L}^{-1}_{B'}$ and we obtain the sequence
\begin{equation*}\label{suc_fsuppEF}
0 \lr T_{B'} \otimes \mathcal{L}^{-1}_{B'}  \to (\mathcal{L}^{-1}_{B'})^{\oplus 2} \to \mathcal{N}_{\sigma_4 \circ \sigma_3} \otimes \mathcal{L}^{-1}_{B'} \to 0.
\end{equation*}
Considering the induced long exact sequence in cohomology,  by \eqref{eq.LB} 

\begin{equation} \label{eq_anti0}
h^0\big(T_{B'} \otimes \mathcal{L}^{-1}_{B'} \big) =0
\end{equation}
\begin{equation} \label{eq_anti}
h^1\big(T_{B'} \otimes \mathcal{L}^{-1}_{B'} \big) = h^0\big(\mathcal{N}_{\sigma_4 \circ \sigma_3} \otimes \mathcal{L}^{-1}_{B'}\big)  
\end{equation}
and
\begin{equation}\label{eq_anti2}
h^2\big(T_{B'} \otimes \mathcal{L}^{-1}_{B'} \big) = h^1\big(\mathcal{N}_{\sigma_4 \circ \sigma_3} \otimes \mathcal{L}^{-1}_{B'}\big) +2\end{equation}

\begin{lem}\label{lem_anti} It holds
\[
 h^0\big(\mathcal{N}_{\sigma_4 \circ \sigma_3} \otimes \mathcal{L}^{-1}_{B'}\big) =2.
\]
\end{lem}
  \begin{proof} Recall that we set $E'=\sigma^*_4E$. Let us consider the exact sequence \eqref{eq.blow-up.tangent}, it lifts on $B'$ as
  \[
  0 \to \sigma_4^*T_B \to \oo_{B'}^{\oplus 2} \to \sigma_4^*\oo_{E}(-E)\to 0.
  \]  
We put  this last exact sequence together with \eqref{eq.blow-up.tangent4} as respectively the  middle horizontal sequence and the first vertical sequence  in a diagram. Chasing  the diagram we  obtain the following  {\small
\begin{equation}\label{big_diagram}
\begin{xy}
\xymatrix{
& 0   \ar@{->}[d]& 0  \ar@{->}[d] \\
& T_B \ar@{=}[r] \ar@{->}[d] & T_B \ar@{->}[d] \\
0 \ar@{->}[r] & \sigma^*_4(T_{B_0}) \ar@{->}[r] \ar@{->}[d] & \oo^{\oplus 2}_B \ar@{->}[d]\ar@{->}[r] & \sigma_4^*(\oo_{E}(-E)) \ar@{=}[d] \ar@{->}[r] & 0 \\
0  \ar@{->}[r]  & \oo_F(-F) \ar@{->}[d] \ar@{->}[r] &  \mathcal{N}_{\sigma_4 \circ \sigma_3} \ar@{->}[d]   \ar@{->}[r] & \sigma_4^*(\oo_{E}(-E))   \ar@{->}[r]  &  0 \\
& 0& 0
 }
\end{xy}
\end{equation}}
which is  a diagram with exact rows and columns.
Let us look at the last horizontal sequence. Recall that $F \cong \PP^1 \cong E'$, thus  $\oo_{F}(-F) \cong \oo_{\PP^1}(1)$ and  $\oo_{E'}(-E') \cong  \oo_{\PP^1}(1)$. Moreover, the sheaf $\sigma_4^*(\oo_{E}(-E))$ is locally free and it is supported on  $F \cup E'$. Its restriction to the irreducible components are
\begin{align*} 
\sigma_4^*(\oo_{E}(-E))|_{E'}&\cong \oo_{\PP^1}(1), &
\sigma_4^*(\oo_{E}(-E))|_F&\cong \oo_{\PP^1}. 
\end{align*}
We tensor the last horizontal sequence in \eqref{big_diagram} by $\mathcal{L}_B^{-1} \cong \oo_{B'}(R+E'+C_1)$ and we get
\[
0 \to \oo_{\PP^1}(-1) \to  \mathcal{N}_{\sigma_4 \circ \sigma_3} \otimes \mathcal{L}_{B'}^{-1} \to \sigma_4^*(\oo_{E}(-E)) \otimes  \mathcal{L}_{B'}^{-1} \to 0.
\]
The long exact sequence in cohomology yields
\[
H^i(\mathcal{N}_{\sigma_4 \circ \sigma_3} \otimes \mathcal{L}_{B'}^{-1}) \cong H^i(\sigma_4^*(\oo_{E}(-E)) \otimes  \mathcal{L}_{B'}^{-1}), \qquad \forall \, i.
\]
By the intersection computation
 \begin{equation*}\label{eq_intersectionL}
(R+E'+C_1)E'=-2
\textrm{ and } 
(R+E'+C_1)F=4
\end{equation*}
the sheaf $\sigma_4^*(\oo_{E}(-E)) \otimes  \mathcal{L}_{B'}^{-1}$ is a locally free sheaf on $E' \cup F$ which has degree $-1$ on $F$ and degree $2$ on $E'$.
 Hence its global sections vanish on $F$ and, fixing an isomorphism $E' \cong \PP^1$, correspond to the sections of $H^0(\oo_{\PP^1}(2))$ which vanish  on the point $E' \cup F$. 
 
 Thus
\[
 H^0\big(\mathcal{N}_{\sigma_4 \circ \sigma_3} \otimes \mathcal{L}^{-1}_{B'}\big) \cong H^0(\oo_{\PP^1}(1)) \cong \CC^2.
\]
  \end{proof}


\begin{rem}\label{rem_iso} Let $q\in S$ be the point blown-up by $S' \rightarrow S$. 
The short exact sequence obtained pushing forward  \eqref{suc.def.S} produces a cohomology exact sequence 
\[
0 \rightarrow T_q  S \rightarrow H^1(S',T_{S'}) \rightarrow H^1(S,T_S) \rightarrow 0.
\]

Recall that if $\beta \colon S' \longrightarrow B'$ is a finite two to one cover, then $ H^1(S',T_{S'})= H^1(B',\beta_* T_{S'})$  splits as invariant and anti-invariant part.
Since $q$ is an isolated fixed point of the involution induced by the Albanese map, it acts as the multiplication by $-1$ on $T_qS$ and then
the image of $T_qS$ is contained in   $H^1 (S',T_{S'})^- $.
By  ({\it e.g.} Pardini \cite[Lemma 4.2]{Pa91}) we have 
\begin{align*}
 (\beta_*T_{S'})^+ & \cong T_{B'}(-\log (R+E'+C_1))   &(\beta_*T_{S'})^- & \cong  T_{B'} \otimes \mathcal{L}^{-1}_{B'} 
\end{align*}
By the Lemma \ref{lem_anti} and \eqref{eq_anti} then $h^1 (\beta_*T_{S'})^-  =2$, and so $T_qS$ maps isomorphically onto  $H^1 (S',T_{S'})^- $.

In particular the map
\[
 H^1(B',T_{B'}(-\log (R+E'+C_1))) \rightarrow  H^1(S, T_S),
\]
is an isomorphism.
\end{rem}

\begin{rem}\label{rem_InvEAntiInv}
Corollary \ref{cor_deformation} apply to the blow-up $\sigma_4 \colon B' \rightarrow B$ with $D'=R+E'+C_1$ since all required rigidites have been proved in Lemma \ref{Lem_Normal}.

So the natural map
\[
H^1\big(T_{B'}(-\log (R+E'+C_1)) \big) \hookrightarrow H^1\big(T_{B} \big)
\]  
is injective. Since the map $ H^1\big(T_{B} \big) \rightarrow  H^1\big(T_{A} \big)$ is an isomorphism, 
\[
H^1\big(T_{B'}(-\log (R+E'+C_1)) \big) \hookrightarrow H^1\big(T_{A} \big).
\]  
is injective as well.

\end{rem}
 Hence we have a commutative diagram
  \[
\begin{xy}
\xymatrix{
H^1(T_{B'}(-\log ((R+E'+C_1))) \ar@{->}[d]^{\cong}    \ar@{->}[r]  & H^1(T_{B'})  \ar@{->>}[d]  \\
H^1(T_S) \ar@{^{(}->}[r] & H^1(T_A) . 
 }
\end{xy}
\]  
The left vertical map is an isomorphism by Remark \ref{rem_iso}. The composition of  the top horizontal arrow and the right vertical arrow is the map in Remark \ref{rem_InvEAntiInv}, 
so injective, and therefore the lower horizontal map is injective.

\begin{prop} \label{prop_H1tan} It holds 
\[
h^1(T_S)=2.
\]
\end{prop}
\begin{proof} The image of the  map $H^1(T_S) \rightarrow H^1(T_A)$ is contained in the hyperplane $H$ of  Lemma \ref{lem_prodElliptic}, since the Albanese variety of every surface of general type is an abelian variety.

We proved that $H^1(T_{B'}(-\log ((R+E'+C_1))) \cong H^1(T_S)$ and the induced map  $\varphi \colon  H^1(T_{B'}(-\log (R+E'+C_1)))  \to  H^1(T_A)$ is injective. So it is enough to prove
$\dim Im(\varphi)=2$

The function $\varphi$ factorizes as in the following commutative diagram.
  \[
\begin{xy}
\xymatrix{
H^1(T_{B'}(-\log (R+E'+C_1)) \ar@{->}[d]    \ar@{->}[r]  &  H^1(T_A)  \\
H^1(T_{B'}(-\log (C_1))  \ar@{->}[r]  & H^1(T_A( - \log C_1))   \ar@{->}[u]_{\varepsilon}
 }
\end{xy}
\]  
where $C_1$ is the elliptic curve in Figure \ref{Fig4}. 
We recall that $A$ is isogenous to the product of two elliptic curves $T_1 \times T_2$ and $C_1$ is a fibre of the induced elliptic fibration $f_2$ on $T_2$. \\ 
So the image of $\epsilon$ is contained in $H_2$. Then $Im(\varphi) \subset H \cap H_2$ has,  by  Lemma \ref{lem_prodElliptic}, dimension at most 2. On the other hand it is at least $2$ by Proposition \ref{prop_moduli_dim}, and therefore it equals $2$. 
\end{proof}
  
\begin{prop} \label{prop.moduli}  The following holds: for all $j \in\{1,2,4\}$
$\mathcal{M}_j $ is 
a generically smooth irreducible component of the moduli space of the surfaces of general type of
dimension $2$.
\end{prop}
\begin{proof}
We have shown that $\mathcal{M}_j$ is  irreducible of dimension $2$ in Proposition  \ref{prop_irred}.
Then by Proposition  \ref{prop_H1tan} $Def(S)$ is smooth of dimension $2$ at each point. It follows that $\mathcal{M}_j$ is an irreducible component, and that this component is generically smooth.
\end{proof}

\bigskip

Matteo Penegini, Universit\`a degli Studi di Genova, DIMA Dipartimento di Matematica, I-16146 Genova, Italy \\
\emph{e-mail} \verb|penegini@dima.unige.it|

\medskip

Roberto Pignatelli Universit\`a degli Studi di Trento, Dipartimento di Matematica,
I-38123 Trento,  Italy \\
\emph{e-mail} \verb|Roberto.Pignatelli@unitn.it|


\begin{thebibliography} {9}
%
\bibitem[BaCaPi06]{BaCaPi06}
I. Bauer, F. Catanese, R. Pignatelli: Complex surfaces of general type: some recent progress, in \emph{Global methods in complex geometry}, Springer-Verlag (2006), 1-58.
%
\bibitem[BCGP]{BCGP}
I. Bauer, F. Catanese, F. Grunewald, R. Pignatelli,
\textit{Quotients of products of curves, new surfaces with $p_g=0$ and their fundamental groups}.  
Amer. J. Math. {\bf 134} (2012), no. 4, 993–1049
%
\bibitem[Bea82]{Be82}
A. Beauville: {\it L'inegalit{\'e} $p_g \geq 2q-4$ pour les surfaces de type g{\'e}n{\'e}rale}, Bull. Soc. Math. de France
$\boldsymbol{110}$ (1982), 343-346.
%
\bibitem[BL04]{BL04}
C. Birkenhake, H. Lange, \textit{Complex abelian varieties}.
Grundlehren der Mathematischen Wissenschaften, Vol \textbf{302},
Second edition, Springer-Verlag, Berlin, 2004.
%
\bibitem[BO20]{ortega}
P Bor\`owka, A Ortega, \textit{Klein coverings of genus 2 curves}
 Trans. Amer. Math. Soc. {\bf 373} (2020), no. 3, 1885–1907
%
\bibitem[CanFr15]{CanFr15}
N. Cancian, D. Frapporti: {\it On semi-isogenous mixed surfaces}, Mathematische Nachrichten, 2018, {\bf 291}, pp. 264--283.
%
\bibitem[Cat84]{Ca84}
F. Catanese, \textit{On the moduli spaces of surfaces of general type}. Jour. Diff. Geom. 19, 2 (1984), 483–515.
%
\bibitem[Cat89]{Ca89}
F. Catanese, \textit{Everywhere non reduced moduli spaces}.
Invent. Math. \textbf{98} (1989), 293-310.
%
\bibitem[Cat90]{Ca90}
F. Catanese, \textit{Footnotes to a theorem of I. Reider}.
Algebraic geometry (L'Aquila, 1988), Lecture Notes in Math., Vol
\textbf{1417}. Berlin, 1990, 67--74.
%
\bibitem[Cat91]{Ca91}
F. Catanese, \textit{Moduli and classification of irregular
Kaehler manifolds (and algebraic varieties) with Albanese general
type fibrations}. Invent. Math. \textbf{104} (1991), 263-289.
%
\bibitem[CaCiML98]{CaCiML98}
F. Catanese, C. Ciliberto, M. M. Lopes, {\it On the classification of irregular surfaces of general type with non birational bicanonical
map}. Trans. of the Amer. Math. Soc. $\boldsymbol{350}$ (1998), 275--308.
%
\bibitem[Cat00]{Ca00}
F. Catanese, \textit{Fibred surfaces, varieties isogenous to a product and related moduli spaces}, Amer. J. Math. {\bf 122} (2000), 1–44
%
\bibitem[Cat11]{Ca11}
F. Catanese, \textit{A superficial working guide to deformations
and moduli},  Handbook of moduli, in honour of David Mumford, Vol. I, 161–215, Adv. Lect. Math. (ALM), 24, Int. Press, Somerville, MA, (2013).
%
\bibitem[Cat15]{Ca15}
F. Catanese, \textit{Topological methods in moduli theory} Bull. Math. Sci. {\bf 5}, No. 3, 287--449 (2015).
%
\bibitem[Deb81]{Deb81}
O. Debarre: \textit{ Inegalit{\'es} num{\'e}riques pour les surfaces de
type g{\'e}n{\'e}rale}, Bull. Soc. Math. de France
 $\boldsymbol{110}$ (1982), 319-346.
%
\bibitem[DS05]{DS05}
F. Diamond, J. Shurman, \textit{A First Course in Modular Forms}, Springer GTM 228, 2005.
%
\bibitem[GH94]{GH}
P. Griffith, J. Harris {\it Principles of Algebraic Geometry} Wiley Classics Library 1994.
%
\bibitem[Har79]{Har79}
J. Harris, \textit{Galois groups of enumerative problems}. Duke
Math. J. \textbf{46} (1979), 685--724.
%
\bibitem[HP02]{HP02}
C. Hacon, R. Pardin, {\it Surfaces with $p_g=q=3$}, Trans. Amer. Math. Soc. $\boldsymbol{354}$ (2002), 2631--2638.
%
\bibitem[Har10]{H}
R. Hartshorne, {\it Deformation Theory}, Springer GTM 257, 2010.
%
\bibitem[HKW93]{HKW} K. Hulek, C. Kahn, S.H. Weintraub, {\it Moduli spaces of Abelian surfaces: compactification, degenerations,
and Theta functions}. Walter de Gruyter 1993.
%
\bibitem[MP12]{MP12}
M. Mendes Lopes, R. Pardini, {\it The geography of irregular surfaces}. 
Current developments in algebraic geometry, 349–-378, Math. Sci. Res. Inst. Publ. $\boldsymbol{59}$, Cambridge Univ. Press (2012).
%
\bibitem[Par91]{Pa91}
R. Pardini, \textit{Abelian covers of algebraic varieties}. J. Reine
Angew. Math. \textbf{417} (1991), 191--213.
%
\bibitem[Pen09]{Pe09}
M. Penegini, \textit{The classification of isotrivial fibred
surfaces with $p_g=q=2$}, with an appendix by S. Roellenske. Collect. Math. \textbf{62}, No. 3, (2011), 239--274.
%
\bibitem[Pen13] {P}
M. Penegini, \textit{On the classification of surfaces of general type with $p_g=q=2$} , Boll. Uni. Mat. Ital. VI (2013) 549--563
%
\bibitem[PePol13a] {PP}
M. Penegini, F. Polizzi, \textit{On surfaces with $p_g=q=2, K^2=6$ and Albanese map of degree $2$}, Canad. J. Math. {\bf 65} (2013), 195--221
%
\bibitem[PePol13b] {PP10}
M. Penegini, F. Polizzi, \textit{On surfaces with $p_g = q = 2, K^2 = 5$ and Albanese map of degree $3$}. Osaka Journal of Mathematics {\bf 50} (2013), pp. 643 -- 686.
%
\bibitem[PePol14]{PePol14}
M. Penegini, F. Polizzi, \textit{A new family of surfaces with $p_g=q=2$
and $K^2=6$ whose Albanese map has degree $4$}, J. London Math. Soc. $\boldsymbol{90}$ (2014), 741--762
%
\bibitem[PiPol17]{PiPol17}
R. Pignatelli, F. Polizzi, {\it A family of surfaces with $p_g=q=2$, $K^2=7$ and Albanese map of degree $3$}
Mathematische Nachrichten, 2017, {\bf 290}, pp. 2684--2695
%
\bibitem[Pir02]{Pir02}
G.P. Pirola, {\it Surfaces with $p_g=q=3$}, Manuscripta Math.
 $\boldsymbol{108}$ no. 2 (2002), 163--170
 %
 \bibitem[PRR20]{PRR20}
F. Polizzi, C. Rito, X. Roulleau, {\it A pair of rigid surfaces with $p_g = q = 2$ and $K^2 = 8$ whose universal cover is not the bidisk}. International Mathematics Research Notices, Volume {\bf 2020}, (2020).
 %
\bibitem[Rit18] {rito}
C. Rito \textit{New surfaces with $K^2=7$  and $p_g=q\leq 2$},  Asian J. Math. {\bf 22} (2018), no. 6, 1117–1126.
%
\bibitem[Sch68] {S68}
M.Schlessinger, {\it Functors of Artin rings}. Trans. Amer. Math. Soc. {\bf 130} (1968) 208--222.
%
\bibitem[Ser06]{Se06}
E. Sernesi, \textit{Deformations of Algebraic Schemes}.
Grundlehren der Mathematischen Wissenschaften, Vol \textbf{334},
Springer-Verlag, Berlin, 2006.
%
\bibitem[Zuc03]{Z03}
F. Zucconi, \textit{Surfaces with $p_g=q=2$ and an irrational
pencil}. Canad. J. Math. \textbf{55} (2003), no. 3, 649--672.

\end{thebibliography}
\end{document}